\documentclass[12pt,a4paper]{article}

\usepackage[utf8]{inputenc}
\usepackage[english]{babel}

\usepackage[T1]{fontenc}
\usepackage{upgreek}

\usepackage[fleqn]{amsmath} 
\renewcommand{\vec}[1]{{\mbox{\mathversion{bold}\ensuremath{#1}}}}

\usepackage{amsthm}		
\newtheoremstyle{nicetheoremstyle}
  {2\topsep}
  {2\topsep}
  {\itshape}
  {}
  {\large\scshape}
  {:}
  {\topsep}
  {}
\newtheoremstyle{nicedefinitionstyle}
  {\topsep}
  {\topsep}
  {\itshape}
  {}
  {\scshape}
  {:}
  {\topsep}
  {}

\theoremstyle{nicetheoremstyle}
\newtheorem{theorem}{Theorem}
\theoremstyle{nicedefinitionstyle}
\newtheorem{definition}[theorem]{Definition}
\newtheorem{proposition}[theorem]{Proposition}

\newtheorem{lemma}[theorem]{Lemma}

\usepackage{amsfonts}
\usepackage{amssymb}

\usepackage{nicefrac}	

\usepackage{fourier}

\usepackage{graphicx}
\usepackage{wrapfig}
\usepackage{pgfplots} 

\usepackage{booktabs}

\title{Sparse Error Localization in Complex Dynamic Networks}

\author{Dominik Kahl, Andreas Weber, Maik Kschischo}

\begin{document}

\maketitle

\begin{abstract}
	Understanding the dynamics of complex systems is a central task in many different areas
	ranging form biology via epidemics to economics and engineering. Unexpected behaviour 
	of dynamic systems or even systems failure is sometimes difficult to comprehend. 
	Such unexpected dynamics 
	can be caused by systematic model errors, unknown inputs 
	from the environment and systems faults. Localizing the root cause of these errors or faults and 
	reconstructing their dynamics is only possible if the measured outputs of the system are 
	sufficiently informative. 
	Here, we present a mathematical theory for the measurements required to localize the position 
	of error sources in large dynamic networks. We assume, that faults or errors occur at 
	a limited number of positions in the network. This sparsity assumption facilitates the  
	accurate reconstruction of the dynamic time courses of the errors by solving a 
	convex optimal control problem. For cases, where 
	the sensor measurements are not sufficiently informative to pinpoint the error position 
	exactly,
	we provide methods to restrict the error location to a smaller subset of network nodes. 
	We also suggest strategies to  efficiently select additional measurements for narrowing down the 
	error location. 
\end{abstract}

\clearpage\begin{Large} \textbf{E}\end{Large}rrors in complex dynamic systems can be difficult to find. The presence of an error or a fault
is detected by the system's unexpected behaviour. Localizing the root cause
of the error can, however, be much harder or even impossible. Modern technical devices like cars
or planes have dedicated sensors for fault diagnosis or fault isolation. These designs 
are based on fault detection theory \cite{isermann_fault_diagnosis_2011, blanke_diagnosis_2016}, 
an important research field in control theory. During the design process, engineers specify 
a set of components whose failure or malfunction must be detected and isolated. To achieve the detectability 
and isolationability of these prespecified faults, they place sensors with output patterns
uniquely indicating a specific type of error. However, there is a trade-off: Increasing the number
of errors which can be uniquely isolated requires a larger number of sensors. Indeed, one can show from 
invertibility theory, that the number of error signals which can uniquely be distinguished is never
smaller than the number of sensors required~\cite{silverman_inversion_1969, sain_invertibility_1969, fliess_note_1986, wey_rank_1998, kahl_structural_2019}. 

Finding the root cause of errors or faults is also a crucial task in the analysis of naturally evolved systems such as ecological food webs, biochemical reaction networks or infectious disease models. For example, a protein biomarker
like the prostate-specific antigen (PSA) can indicate a disease, here prostate cancer. But, there
can be other causes for increased levels of PSA including older age and hyperplasia~\cite{patel_new_2017}.  A reliable 
diagnosis requires additional tests, or, stated in engineering terms, the isolation of the fault requires more
sensors. 

An important difference between engineered and natural or evolved systems is the degree of uncertainty 
about the interactions between the system's state variables. The couplings between the states of a biochemical reaction network or an ecological food web are usually only partially known. These uncertainties 
make the development of a useful mathematical model difficult; missing, spurious or misspecified interactions generate structural model errors. Moreover, real world dynamic systems are open, i.e. they receive unknown inputs from their environment. Together, these uncertainties often result in 
incorrect model predictions. Detection of genuine faults is even harder, if model errors, 
unknown inputs and faults are interfering. Also, the engineering approach of specifying certain 
faults in advance is difficult to realize in the case of an incomplete or erroneous model.  

In this paper, we provide a mathematical theory for the localization and reconstruction of 
\textit{sparse} faults  and model errors in complex dynamic networks described by ordinary differential equations (ODEs). Sparsity means here, that there is a maximum number of state variables (state nodes) $k$ affected by an error. Typically, $k$ is much smaller than the total number of state nodes $N$. In contrast to fault isolation approaches~\cite{isermann_fault_diagnosis_2011, blanke_diagnosis_2016}, we do not require the a priori specification of certain types of faults, but we allow for the possibility that each state node in the network 
can potentially be targeted by errors (or faults). The sparse error assumption is often realistic in both the model error and the fault detection context. Faults often affect only a small number of nodes in the network, because the simultaneous failure of several components in a system is unlikely to occur spontaneously. For example, a hardware error or a network failure usually occurs at one or two points in the system, unless the system has been deliberately attacked simultaneously at several different positions. Similarly, gene mutations often affect a restricted number
of proteins in a larger signal transduction or gene regulatory network. In the context of model error localization and reconstruction, the sparsity assumption implies that the model is incorrect only at a limited number of positions or, alternatively, that small inaccuracies are ignored and that we focus only on the few (less than $k$) state variables with grossly misspecified governing equations.
Throughout this paper, we will exploit the fact that faults, model errors and interactions with the environment can 
all mathematically be represented as unknown inputs to the system~\cite{mook_minimum_1987, moreno_observabilitydetectability_2012,engelhardt_learning_2016, engelhardt_bayesian_2017, kahl_structural_2019}. Thus, we
use model error, fault and unknown input as synonyms. 

An example for sparse error localization is provided in Fig.~\ref{fig:Fig1}. The nodes of the network in~Fig.~\ref{fig:Fig1}(a) represent the $N=30$ state variables $\vec{x} = (x_1,\ldots, x_{30})^T$ of the dynamic system and the edges their interactions. The squares indicate the subset of $P=10$ states which are directly measured with output time courses  $\vec{y}(t)=\left( y_1 (t), \ldots, y_{10} (t)\right)$ plotted in 
Fig.~\ref{fig:Fig1}(b). In our simulation, we randomly chose state node $x_6$  to be affected by an unknown input $$\vec{w}^* (t)=(0,\ldots,0,w_6^*(t),0,\ldots,0)^T \, ,$$ 
as highlighted by the wiggly arrow in Fig.~\ref{fig:Fig1}(a). In reality, the location of this error would be unknown and only 
the deviations of the output measurements $\vec{y}(t)$ from the expected behaviour would indicate the presence of an error somewhere in the network. Our first result in this paper is a criterion for the 
localizability of an unknown input in nonlinear systems from output measurements. Based on this, 
we can decide that a single unknown input ($k=1$) can be localized 
and reconstructed from the output data $\vec{y}(t)$. The estimate $\hat{\vec{w}}(t)$ in Fig.~\ref{fig:Fig1}(c) provides a satisfying reconstruction of the  true input $\hat{\vec{w}}(t)$.
This illustrates our second result: An optimization criterion for the reconstruction of $k$-sparse model errors and conditions guaranteeing the accuracy of the reconstruction for linear systems. 

In many real world systems, there are practical restrictions on the number or the location of measurements. This restricted set of sensor nodes might render the system non-localizable for a given sparsity $k$, i.e. the state nodes (at most $k$) affected by the error or fault cannot exactly be determined from the available output. For example, it is often difficult to
measure the dynamic time course of a large number of proteins and other molecules in a biochemical reaction network
and the root cause for observed errors (e.g. diseases) might be impossible to find. Then,  it might still be useful to at least narrow down the 
location of the errors to subsets of states. As a further result, we present here a coherence measure which quantifies, how difficult it is to distinguish different states as root causes of a $k$-sparse error. This coherence measure can be used to cluster states 
into indistinguishable subsets and to subsequently identify those subsets targeted by model errors. 

Once the subsets of the states targeted by errors are known, additional output data are needed to further 
narrow down the exact error location. We provide an efficient sensor node placement strategy for reducing 
the uncertainty about the error location. In combination with coherence based output clustering,
this sensor selection strategy can be iterated to restrict the exact error location 
to smaller and smaller subsets until the exact position is isolated. 
	\begin{figure*}
		\centering
		\includegraphics[width=\textwidth]{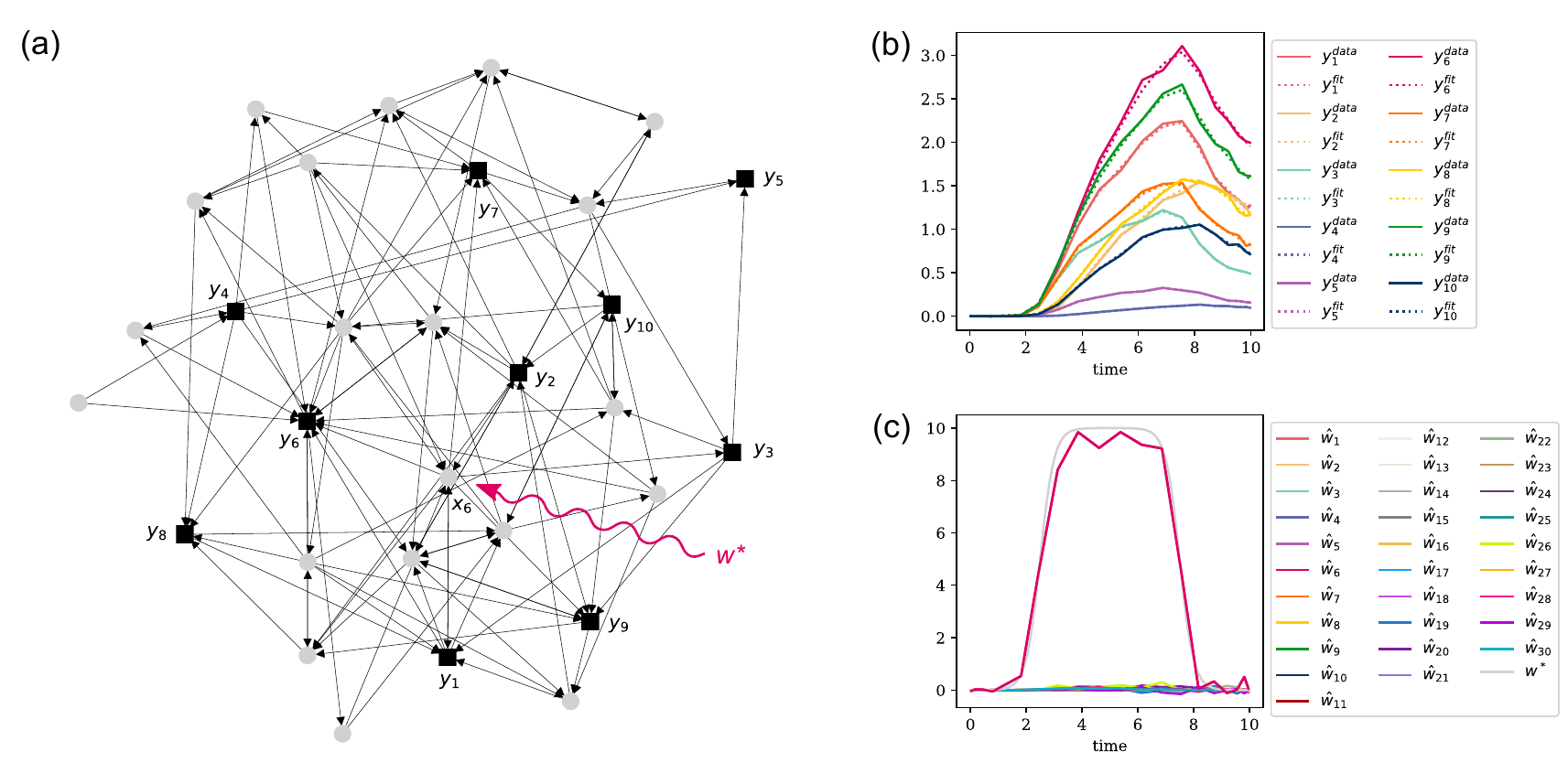}
		\caption{Reconstruction of a sparse unknown input.  (a) The influence graph of 
			a linear dynamic system with $N=30$ states. The nodes correspond to the state variables and 
			the edges indicate their interactions. The simulated error signal $\vec{w}^* (t)=(0,\ldots,
			0,w_6^*(t),0,\ldots,0)^T$ targets the state variable $x_6$. The squares indicate the $P=10$
			sensor nodes  providing the output $\vec{y}=(y_1,\ldots, y_{10})^T$ (b) 
			The  measured output data 
			$\vec{y}^{\text{data}}(t)=(y^{\text{data}}_1(t),\ldots, y^{\text{data}}_{10}(t))^T$ (solid lines) can be fitted 
			(dashed lines) by the output $\hat{\vec{y}}(t)$
			corresponding to the solution $\hat{\vec{w}}(t)$ (see (c)) of the 
			convex optimal control problem in~\eqref{eq:optimization}. (c) This estimate $\hat{\vec{w}}(t)$ 
			simultaneously reconstructs the true unknown input $\vec{w}^* (t)$. One can see, that 
			among the thirty inputs the node $i=6$ ($w_6$) was localized as the root cause of the error.
			}
		\label{fig:Fig1}
	\end{figure*}
\section*{Background}
\subsection*{Open dynamic systems with errors and faults}
We consider input-output systems of the form 
	\begin{equation}
		\begin{aligned} \label{eq:DynamicSystem}
		\dot{\vec{x}}(t) &= \vec{f}(\vec{x}(t)) + \vec{w}(t) \\
		\vec{x}(0) &= \vec{x}_0	 \\
		\vec{y}(t) &= \vec{c}(\vec{x}(t))
		\end{aligned}
	\end{equation}
where $\vec{x}(t) \subseteq \mathbb{R}^N$ denotes the state of the system at time $t\in [0,T]$ and $\vec{x}_0\in\mathbb{R}^N$ is the initial state. The vector field $\vec{f}$ encodes the model of the system and is assumed to be Lipshitz. The function $\vec{c}:\mathbb{R}^N \to \mathbb{R}^P$  describes the measurement process and maps the system state $\vec{x}$ to the directly
observable output~$\vec{y}$. 

Model errors or faults are represented as unknown input functions $$\vec{w}:[0,T] \to \mathbb{R}^N \, .$$ 
This ansatz incorporates all types of errors, including missing and wrongly specified interactions, parameter errors~\cite{mook_minimum_1987, kahm_potassium_2012, schelker_comprehensive_2012, engelhardt_learning_2016, engelhardt_bayesian_2017, tsiantis_optimality_2018} as well as faults~\cite{isermann_fault_diagnosis_2011, blanke_diagnosis_2016} and  unobserved inputs from the environment~\cite{kahl_structural_2019}. 

The system in~\eqref{eq:DynamicSystem} can be seen as an input-output map $\Phi:\mathcal{W}\to \mathcal{Y}, \vec{w}\mapsto \vec{y}$. The input space $\mathcal{W}=\mathcal{W}_1 \oplus \ldots \oplus \mathcal{W}_N$ is assumed to be the direct sum of suitable (see below) function spaces $\mathcal{W}_i,\,i=1,\ldots,N$. For zero errors $\vec{w} \equiv \vec{0}$ (i.e. $\vec{w}(t) = \vec{0} \, \forall t\in [0,T]$) we call the system~\eqref{eq:DynamicSystem} a closed dynamic system.  Please note that we do not exclude the possibility of known inputs for control, but we suppress them from our notation. 

An error or fault can be detected from the residual 
\begin{equation}
	\vec{r}(t) := \vec{y}^\text{data}(t)-\vec{y}^{(0)}(t)
\end{equation}
between the measured output data $\vec{y}^\text{data}(t)$ and the output $\vec{y}^{(0)}(t)=\Phi (\vec{0}) (t)$ of the closed system. To infer the model error $\vec{w}(t)$ we have to solve the equation
	\begin{equation}
		\Phi (\vec{w}) = \vec{y}^\text{data}   \label{eq:DataProblem} 
	\end{equation}
for $\vec{w}$. In general, there can be several solutions to the problem~\eqref{eq:DataProblem}, unless we either measure the full state of the system or we restrict the set of unknown inputs $\vec{w}$. 
In fault detection applications~\cite{isermann_fault_diagnosis_2011, blanke_diagnosis_2016}, the restriction is given by prior assumptions about the states which are targeted by errors. We will 
use a sparsity assumption instead. For both cases, we need some notation: Let $\mathcal{N}=\{1,2,\ldots, N\}$ be the index set of the $N$ state variables and $S \subseteq \mathcal{N}$ be a subset with complement $S^c = \mathcal{N}\setminus S$. By $\vec{w}_{S} (t)$ we indicate the 
vector function obtained from $\vec{w}(t)$ by  setting the entries $(\vec{w}_S)_i$ with $i\in S^c$ to the zero function. If $S$ is of minimal cardinality and $\vec{w}_{S} = \vec{w}$
we call $S$ the \textit{support} of $\vec{w}$. The corresponding restriction on the input space is defined via 
\begin{equation}
	\mathcal{W}_S := \left\{\vec{w}\in \mathcal{W}\left| \, \text{supp}\,\vec{w} \subseteq S \right.\right \}\,.
\end{equation} 
Thus, $S$ characterizes the states $x_i$ with $i\in S$ which can \textit{potentially} be
affected by a non-zero unknown input $w_i$. We will also refer to $S$ as the set of input or source nodes. The restricted input-output map $\Phi_S:\mathcal{W}_S \to \mathcal{Y}$ is again given by~\eqref{eq:DynamicSystem}, but all input components $\vec{w}_i$ with $i\not \in S$ are restricted to be 
zero functions.  

Now, we can formally define invertibility \cite{silverman_inversion_1969,sain_invertibility_1969}:
	%
	%
	\begin{definition}\label{def:invertible} 
		The system~\eqref{eq:DynamicSystem} with input set $S$ and input-output map $\Phi$ is called 
		\textbf{invertible}, if for two 
		different solutions $\vec{w}^{(1)},\vec{w}^{(2)} \in \mathcal{W}_S$ of \eqref{eq:DataProblem} and for any data set 
		$\vec{y}^\text{data}:[0,T]\to \mathbb{R}^P$ we have 
		\begin{equation}	
			\vec{w}^{(1)} (t) -  \vec{w}^{(2)} (t)=\vec{0}
		\end{equation}					
		almost everywhere in $[0,T]$.		
	\end{definition}
In other words, invertibility guarantees that \eqref{eq:DataProblem} with an input set $S$ has only one solution $\vec{w}^*$ (up to differences of measure zero), which corresponds to the true model error. In the following, we mark this true model error with an asterisk, while $\vec{w}$ without asterisk denotes an indeterminate input function.	
		
\subsection*{Structural Invertibility and Independence of Input Nodes}
There are several algebraic or geometric conditions for invertibility~\cite{silverman_inversion_1969,sain_invertibility_1969,fliess_note_1986,fliess_nonlinear_1987,basile_new_1973}, which are, however, difficult to test for large systems and require exact knowledge of the systems equations~\eqref{eq:DynamicSystem}, including all the parameters. \emph{Structural invertibility} of a system 
is prerequisite for its invertibility and can be decided from a graphical 
criterion \cite{wey_rank_1998}, see
also Theorem~\ref{theorem:structuralinvertibility} below. Before, we define
the influence graph~(see e.g.~\cite{liu_control_2016})
	%
	%
	\begin{definition} \label{def:influencegraph}	
		The \textbf{influence graph}
		$g=(\mathcal{N},\mathcal{E})$ of the system \eqref{eq:DynamicSystem} is a digraph, where the set 
		of nodes~$\mathcal{N}=\{1,2,\ldots,N\}$
		represents the $N$ state variables \\$\vec{x}=(x_1,\ldots,x_N)$ and the set of directed edges
		$\mathcal{E}=\{i_1 \to l_1, i_2 \to l_2,\ldots\}$ represents the interactions between those
		states in the following way: There is a directed edge $i\to l$ for each pair of state nodes
		$i,l\in \mathcal{N}$ if and only if $\frac{\partial f_l}{\partial x_i} (\vec{x}) \ne 0$ for some
		$\vec{x}$ in the state space $\mathcal{X}$.  		
	\end{definition}

In addition to the set of input nodes $S \subseteq \mathcal{N}$ we define the output nodes $Z\subseteq \mathcal{N}$ of the system in~\eqref{eq:DynamicSystem}. The latter are determined by the measurement function $\vec{c}$. Without restriction of generality we assume in the following that a subset $Z\subseteq \{1,2,\ldots,N \}$ of $P$  state nodes are sensor nodes, i.e.~they can directly be measured, which corresponds to  $c_i(\vec{x})=x_i$ for $i\in Z$. All states $x_l$ with $l \not \in Z$ are not directly monitored. 
   
A necessary criterion for structural invertibility is given by the following graphical condition \cite{wey_rank_1998}:
	%
	%
	\begin{theorem}\label{theorem:structuralinvertibility}
		Let $g=(\mathcal{N}, \mathcal{E})$ be an influence graph and $S, Z \subseteq \mathcal{N}$  be 
		known input and output node sets with cardinality $M=\text{card}\,S$ and 
		$P=\text{card}\,Z$, respectively. If there 
		is a family of directed paths $\Pi=\{\pi_1,\ldots, \pi_M\}$ with the properties
		\begin{enumerate}
			\item each path $\pi_i$ starts in $S$ and terminates in $Z$,
			\item any two paths $\pi_i$ and $\pi_j$ with $i\neq j$ are node-disjoint,
		\end{enumerate}	
		then the system is structurally invertible.
		If such a family of paths exists, we say \textbf{$S$ is linked in $g$ into~$Z$}.
	\end{theorem}
In the Appendix we discuss why we have the strong indication that this theorem provides also a sufficient criterion 
for structural invertibility up to some pathological cases.
	
A simple consequence of this theorem is that for an invertible system, the number $P$ of sensor nodes cannot be smaller than the number of input nodes $M$. This is the reason, why for fault detection the set of potentially identifiable error sources $S$ is selected in advance \cite{blanke_diagnosis_2016}. Without a priori restriction on the set of potential error sources, we would need to measure all states. Please note, that there are efficient algorithms to check, whether a system with a given influence graph $g$ and given input and sensor node set $S$ and $Z$ is invertible~(see \cite{kahl_structural_2019} for a concrete algorithm and references therein). 
\subsection*{Independence of Input Nodes}
If the path condition for invertibility in Theorem~\ref{theorem:structuralinvertibility} 
is fulfilled for a given triplet $(S,g,Z)$
we can decide, whether the
unknown inputs targeting~$S$ can be identified in the given graph $g$ using 
the set of sensor nodes $Z$. 
Without a priori knowledge about the model errors, however, the input set $S$ is unknown as well.
Therefore, we will consider the case that the input set $S$ is unknown in the results section. To this end, we define an independence structure on the union of all possible input sets:
	%
	%
	\begin{definition}\label{def:Gammoid}
		The triple $\Gamma:=(\mathcal{L}, g,Z)$ consisting of an influence graph $g=(\mathcal{N},\mathcal{E})$,
		and \textbf{input ground set} $\mathcal{L}\subseteq \mathcal{N}$, and 
		an output set $Z$ is called a \textbf{gammoid}. A subset $S\subseteq\mathcal{L}$ is understood as an 
		input set. 
		An input set $S$ is called
		\textbf{independent in $\Gamma$}, if $S$ is linked in $g$ into $Z$.
	\end{definition}
The notion of (linear) independence of vectors is well known from vector space theory. For finite
dimensional vector spaces, there is the rank-nullity theorem relating the dimension of the
vector space to the dimension of the null space of a linear map. The difference between the 
dimension of the vector space and the null space is called the rank of the map. 
The main advantage of the gammoid interpretation lies in the following rank-nullity concept:	
	%
	%
	 \begin{definition}\label{def:rank}
			Let $\Gamma=(\mathcal{L},g,Z)$ be a gammoid.
			\begin{enumerate}
				\item The \textbf{rank} of a set $S\subseteq \mathcal{L}$ is the size of the 
				largest independent subset $\tilde{S} \subseteq S$.
				\item The \textbf{nullity} is defined by the rank-nullity theorem 
				\begin{equation}
					\text{rank}\, S + \text{null}\, S = \text{card}\, S \, .
				\end{equation}
			\end{enumerate}
	\end{definition}
Note, that the equivalence of a consistent independence structure and a rank function 
(see definition \ref{def:rank} 1.) as well as the existence of a rank-nullity theorem (see definition \ref{def:rank} 2.) goes back to the early works on matroid theory \cite{whitney_abstract_1935}. It has 
already been shown \cite{perfect_applications_1968}, that the graph theoretical idea of linked sets (see definition \ref{def:Gammoid}) fulfils the axioms of matroid theory and therefore inherits its properties.	
The term gammoid for such a structure of linked sets was probably first used in \cite{pym_linking_1969} 
and since then investigated under this name, with slightly varying definitions. We find the formulation
above to be suitable for our purposes (see also the SI Appendix for more information about gammoids).  
\section*{\label{sec:Results}Results}
Here, we consider the localization problem, where the input set $S$ is unknown. However, we make a 
sparsity assumption by assuming that $S$ is a small subset of the ground set 
$\mathcal{L}\subseteq \mathcal{N}$. 
Depending on the prior information, the ground set can be the set of all state variables $\mathcal{N}$ or a subset. 

The sparsity assumption together with the definition of independence of input nodes in Definitions~\ref{def:Gammoid} and \ref{def:rank} can be exploited to generalize the idea of sparse sensing \cite{donoho_optimally_2003, candes_decoding_2005, donoho_compressed_2006,yonina_c_eldar_compressed_nodate}  to the solution of the dynamic problem \eqref{eq:DataProblem}. Sparse sensing for matrices is a well established field in signal and image 
processing (see e.g. \cite{yonina_c_eldar_compressed_nodate, foucart_mathematical_2013}). There are, however, 
some nontrivial differences: First, the input-output 
map $\Phi$ is not necessarily linear. Second, even if $\Phi$ is linear, it is a compact operator between 
infinite dimensional vector spaces and therefore the inverse $\Phi^{-1}$ is not continuous. 
This makes the inference of unknown inputs $\vec{w}$ an ill-posed problem, even if \eqref{eq:DataProblem} has 
a unique solution~\cite{nakamura_potthast_inverse_2015}.

\subsection*{Sparse Error Localization and Spark}
Definition \ref{def:Gammoid} enables as to transfer the concept of the \emph{spark} \cite{donoho_optimally_2003} to dynamic systems:
	%
	%
	\begin{definition}\label{def:spark}
			Let $\Gamma=(\mathcal{L},g,Z)$ be a gammoid. The \textbf{spark} 
			of~$\,\Gamma$ is defined as the largest integer, such that for each 
			input set $S\subseteq \mathcal{L}$
			\begin{equation}
				\text{card}\, S < \text{spark}\,\Gamma \,\Rightarrow \,
				\text{null}\, S = 0 \, .
			\end{equation}
	\end{definition}
Let's assume we have a given dynamic system with influence graph 
$g=(\mathcal{N}, \mathcal{E})$ and with an output set $Z\subset \mathcal{N}$. In addition, we haven chosen an input 
ground set $\mathcal{L}$. Together, we have the gammoid $\Gamma=(\mathcal{L},g,Z)$. 
The spark gives the smallest number of inputs that are dependent. 
As for the compressed sensing problem for matrices \cite{donoho_optimally_2003}, we can use the spark 
to check, under which condition a sparse solution is unique: 
	%
	\begin{theorem}\label{theorem:spark2}			
			For an input $\vec{w}$ we denote $\Vert \vec{w} \Vert_0$ the number of 
			non-zero components.
			Assume $\vec{w}$ solves \eqref{eq:DataProblem}.
			If 
			\begin{equation}\label{eq:spark_localizability}
				\Vert \vec{w} \Vert_0 < \frac{\text{spark}\,\Gamma }{2} \, ,
			\end{equation}
			then $\vec{w}$ is the unique sparsest solution.
	\end{theorem}
This theorem provides a necessary condition for the localizability of a $k$-sparse error 
in a nonlinear dynamic system. For instance, if we expect an error or input to target a single
state node like in Fig.~\ref{fig:Fig1}(a), we have  $\Vert \vec{w}^* \Vert_0=1$  and we need 
$\text{spark}\,\Gamma \ge 3$ to pinpoint the exact position of the error in the network.
If an edge in the network is the error source, then two nodes are affected and $\Vert \vec{w}^* \Vert_0=2$. Such an error could be a misspecified reaction rate in a biochemical reaction or a cable break in an electrical network. To localize such an error we need $\text{spark}\,\Gamma \ge 5$. 

For smaller networks like the one in Fig.~\ref{fig:Fig1}(a), it is possible to 
exactly compute the spark (Definition~\ref{def:spark}) of a gammoid (Definition \ref{def:Gammoid}) using an combinatorial algorithm iterating over all possible input node sets. However, the computing time grows rapidly with the size of the network. Below we present bounds for the spark, which can efficiently be computed.

\subsection*{Convex Optimization for Sparse Input Reconstruction}
As in compressed sensing for matrices, finding the solution of ~\eqref{eq:DataProblem} with a minimum number of non-zero components $\Vert \vec{w} \Vert_0$ is an NP-hard combinatorial problem. Here, we formulate a convex optimal control problem as a relaxed version of this combinatorial problem. 
We define a Restricted-Isometry-Property (RIP) \cite{candes_decoding_2005} for the input-output operator 
$\Phi$ defined by~\eqref{eq:DynamicSystem} and provide conditions for the exact sparse recovery of errors in linear dynamic systems by solutions of the relaxed problem. As a first step it is necessary to introduce a 
suitable norm promoting the sparsity of the vector of input functions~$\vec{w}(t)$. 

Say, $\mathcal{L}$ is an input ground set of size $L$. The space of input functions 
\begin{equation}
	\mathcal{W} := \bigoplus_{i\in\mathcal{L}} \mathcal{W}_i
\end{equation}
is composed of all function spaces $\mathcal{W}_i$ corresponding to input component $w_i$. 
By definition of the input ground set, $w_i\equiv 0$ for $i\not\in \mathcal{L}$. Thus 
$\mathcal{W}_i=\{0\}$ for $i\not\in\mathcal{L}$.
Assume, that each function space $\mathcal{W}_i=L^p([0,T])$ is a Lebesque space equipped with the 
$p$-norm
\begin{equation}
	\Vert w_i \Vert_p = \left(\int_0^T |w_i(t)|^p dt\right)^{1/p} \, .
\end{equation}
We indicate the vector  
\begin{equation}
	\underline{\vec{w}} := \begin{pmatrix}
		\Vert w_1 \Vert_p \\ \vdots \\ \Vert w_L \Vert_p 
	\end{pmatrix} \, \in\mathbb{R}^L \label{eq:pnorm}
\end{equation}
collecting all the component wise function norms by an underline. 
Taking the $q$-norm in $\mathbb{R}^L$ 
\begin{equation}
	\Vert \underline{\vec{w}} \Vert_q = \left( 
		\underline{w}_1^q + \ldots + \underline{w}_L^q
	\right)^{1/q} 
\end{equation}
of $\underline{\vec{w}}$ yields the $p$-$q$-norm on $\mathcal{W}$ 
\begin{equation}
	\Vert \vec{w} \Vert_{q} : = \Vert \underline{\vec{w}} \Vert_q \, . \label{eq:qnorm}
\end{equation}
The parameter $p$ appears implicitly in the underline. Since our results are valid for all 
$p\in [1,\infty)$, we will suppress it from the notation.   

Similarly, for the $P$ outputs of the system, the output space
\begin{equation}
	\mathcal{Y} = \mathcal{Y}_1 \oplus \ldots \oplus \mathcal{Y}_P \, .
\end{equation}
can be equipped with a $p$-$q$-norm 
\begin{equation}
	\Vert \vec{y} \Vert_{q} : = \Vert \underline{\vec{y}} \Vert_q \, .
\end{equation}
An important subset of the input space $\mathcal{W}$ is the space $\Sigma_k$ of $k$-sparse inputs
	\begin{equation}
		\vec{w}\in \Sigma_k \Rightarrow \Vert \vec{w} \Vert_0 \leq k \, .
\end{equation}

	In analogy to a well known property \cite{candes_decoding_2005} from compressed sensing we define for our 
	dynamic problem:
	%
	%
	\begin{definition}\label{def:RIP}
		The \textbf{Restricted-Isometry-Property} (RIP) of order $2k$ is fulfilled, if there is 
			a constant $\delta_{2k}\in (0,1)$ such that for any two vector functions 
			$\vec{u},\vec{v}\in\Sigma_k$ 
			the inequalities
			\begin{equation}
				(1-\delta_{2k}) \Vert \underline{\vec{u}} - \underline{\vec{v}}\Vert_2^2
				\leq \Vert \underline{\Phi (\vec{u}) }   - \underline{\Phi (\Vec{v}) } \Vert^2_2
			\end{equation}
			and
			\begin{equation}
				 \Vert \underline{\Phi (\vec{u}) } + \underline{\Phi (\vec{v}) } \Vert^2_2
				\leq (1+\delta_{2k}) \Vert \underline{\vec{u}} + \underline{\vec{v}}\Vert_2^2
			\end{equation}		 
			hold.
	\end{definition}
	The reconstruction of sparse unknown inputs can be formulated as the optimization problem
	\begin{equation}
		\text{minimize } \Vert \vec{w} \Vert_0 \text{ subject to } \Vert \Phi(\vec{w})-
		\vec{y}^\text{data} \Vert_2 \leq \epsilon \label{eq:L0}
	\end{equation}
	where $\epsilon > 0$ incorporates uniform bounded measurement noise. A solution 
	$\hat{\vec{w}}$ of this problem will reproduce the data $\vec{y}^\text{data}$ according to 
	the dynamic equations~\eqref{eq:DynamicSystem} of the system with a minimal set of nonzero 
	components, i.e., with a minimal set $S$ of input nodes. As before, finding this minimal 
	input set is a NP-complete problem. Therefore, let us consider the relaxed problem
	\begin{equation}
		\text{minimize } \Vert \vec{w} \Vert_1 \text{ subject to } \Vert \Phi(\vec{w})-
		\vec{y}^\text{data} \Vert_2 \leq \epsilon \, . \label{eq:L1}
	\end{equation}
		The following result implies, that for a linear system of ODEs with $\vec{f}(\vec{x}) = A\vec{x}$ 
	and $\vec{c}(\vec{x}) = C \vec{x}$ with matrices $A \in \mathbb{R}^{N \times N}$ and 
	$C\in \mathbb{R}^{P \times N}$ in~\eqref{eq:DynamicSystem} the optimization problem~\eqref{eq:L1} 
	has a unique solution.
	%
	%
	\begin{theorem} \label{theorem:convex}
		If $\Phi$ is linear, then \eqref{eq:L1} is a convex optimization problem.
	\end{theorem}
	For a given input vector $\vec{w}\in\mathcal{W}$ we define the best $k$-sparse approximation 
	in $q$-norm as \cite{foucart_mathematical_2013}
	\begin{equation}
		\sigma_k(\vec{w})_q := \min_{\vec{u}\in\Sigma_k} \Vert \vec{w} - \vec{u} \Vert_q 
	\end{equation}
	i.e. we search for the function $\vec{u}$ that has minimal distance to the desired function 
	$\vec{w}$ under the condition that $\vec{u}$ has at most $k$ non-vanishing components. 
	If $\vec{w}$ is $k$-sparse itself, then we can choose 
	$\vec{u}=\vec{w}$ and thus the distance between the approximation and the desired function 
	vanishes, $\sigma_k(\vec{w})_q=0$.
	%
	%
	\begin{theorem} \label{theorem:RIP}
			Assume $\Phi$ is linear and the RIP of order $2k$ holds.
			Let $\vec{w}^*$ be the solution of \eqref{eq:L0}. The
			optimal solution $\hat{\vec{w}}$ of \eqref{eq:L1} obeys
			\begin{equation}
				\Vert \hat{\vec{w}}- \vec{w}^* \Vert_2 \leq C_0 \frac{\sigma_k(\vec{w}^*)_1}{
				\sqrt{k}} + C_2 \epsilon 
			\end{equation}
			with non-negative constants $C_1$ and $C_2$~\footnote{Formulas for the constants 
			$C_1$ and $C_2$ can be found in the supplemental material.}.
	\end{theorem}
	Motivated by the latter theorem we define a cost functional
	\begin{equation}
		J[\vec{w}]:= \frac{1}{2} \Vert \Phi(\vec{w}) - \vec{y}^\text{data} \Vert_2^2 + 
		\beta \Vert \vec{w} \Vert_1
	\end{equation}
	with given data $\vec{y}^\text{data}$ and regularization constant $\beta$. 
	The solution of the optimization problem in Lagrangian form
	\begin{equation}
		\text{minimize } J[\vec{w}] \text{ subject to \eqref{eq:DynamicSystem}} \, ,
		\label{eq:optimization}
	\end{equation}
	provides an estimate for the input $\hat{\vec{w}}$, see Fig.~\ref{fig:Fig1} for an example.

\subsection*{Coherence of potential input nodes in linear systems}
	So far we have given theorems for the localizability and for the reconstruction of sparse 
	errors in terms of the spark and RIP. However, computing the spark or checking whether the 
	RIP condition holds are again problems whose computation time grows rapidly with 
	increasing systems size. Now, we present a coherence measure between a pair of state 
	nodes $i,j$ in linear systems indicating how difficult it is to decide whether a detected error
	is localized at $i$ or  at $j$. The coherence provides a lower bound for the spark and 
	can be approximated by an efficient shortest path algorithm. Computing the coherence for each 
	pair of state nodes in the network yields the coherence matrix, which 
	 can be used to isolate a subset of states where the root cause of the error must 
	be located. 
	
	If the system~\eqref{eq:DynamicSystem} is linear, i.e. $\vec{f}(\vec{x})=A\vec{x}$ and 
	$\vec{c}(\vec{x})=C\vec{x}$, we can use the Laplace-transform
	\begin{equation}
		T(s)\tilde{\vec{w}}(s) = \tilde{\vec{y}}(s),\qquad s\in\mathbb{C}
	\end{equation}	
	to represent the input-output map $\Phi_{\mathcal{L}}$ by the $L \times P$-transfer matrix $T(s)$.
	The tilde denotes Laplace-transform. Again, we assume that $w_i \equiv 0$ for all $i\ne \mathcal{L}$ and 	$\tilde{\vec{w}}(s)$ is the 
	vector of Laplace transforms of the components of $\vec{w}$ which are in the ground set $\mathcal{L}$ .   Recall that  $L \le N$ is the number of states in the ground set $\mathcal{L}$ and $P$ the number of 
	measured outputs. As before, $\mathcal{L}=\mathcal{N}$ is still a possible special case. 
	
	  We introduce the input gramian
	\begin{equation}\label{eq:defGramian}
		G(s) := T^*(s) T(s)
	\end{equation}
	where the asterisk denotes the hermitian conjugate. Note, that the input gramian is a $L\times L$ 
	matrix. Assume that we have chosen an arbitrary but fixed numbering of the states in the 
	 ground set, i.e. $\mathcal{L}=\{l_1,\ldots , l_L\}$ is ordered.
	%
	%
	\begin{definition} \label{def:coherence}
			Let $G$ be the input gramian of a linear dynamic system.
			We call 
			\begin{equation} \label{eq:ijcoherence}
				\mu_{ij}(s):=\frac{
				\vert G_{ij}(s) \vert }{\sqrt{G_{ii}(s)G_{jj}(s)}},\qquad s \in \mathbb{C}
			\end{equation}
			the \textbf{coherence function} of state nodes $l_i$ and $l_j$.
			We call 
			\begin{equation}
				\mu(s):= \max_{i\neq j} \mu_{ij}(s) \label{eq:mutual_coherence}
			\end{equation}
			the \textbf{mutual coherence} at $s\in \mathbb{C}$.
	\end{definition}
	Coherence measures to obtain lower bounds for the spark have been used for signal decomposition 
	\cite{donoho_uncertainty_1989} and compressed sensing for matrices \cite{donoho_optimally_2003}. 
	In the next theorem, we use the mutual coherence for linear dynamic systems in a similar way to provide bounds for the spark.
	%
	%
	\begin{theorem}\label{theorem:sparkmu}
		Let $\Gamma=(\mathcal{L},g,Z)$ be a gammoid with mututal coherence $\mu(s)$ at some point 
		$s\in\mathbb{C}$. then
		\begin{equation}\label{eq:sparkmu}
			\text{spark}\, \Gamma \geq \frac{1}{\mu(s)} + 1 \quad \forall s \in \mathbb{C}.
		\end{equation}
	\end{theorem}
	Since \eqref{eq:sparkmu} is valid
	for all values of $s\in \mathbb{C}$, it is tempting to compute $\inf_{s \in \mathbb{C}} \mu(s)$
	to tighten the bound as much as possible. Please note, however, that $\mu(s)$ is not a holomorphic 
	function and thus the usual trick of using a contour in the complex plane and applying the 
	maximum/minimum modulus principle can not be applied~(see e.g.~\cite{sontag_mathematical_1998}). Instead, we will introduce the 
	shortest path coherence, which can efficiently be computed and which can be used in Theorem~\ref{theorem:sparkmu} to 
	obtain lower bounds for the spark. 

\subsection*{Shortest Path Coherence}
	There is a one-to-one correspondence between linear dynamic systems and weighted\footnote{Weights are understood as real 
	constant numbers.} gammoids. 
	The weight of the edge $j\to i$ is defined 
	by the Jacobian matrix
	\begin{equation}
		F(j \to i) := \frac{ \partial f_i (\vec{x}) }{\partial x_j } \, 
	\end{equation}
	and is constant for a linear system. We extend this definition to sets of paths in the following way: Denote by $\pi = (i_0 \to i_1 \to \ldots \to i_{\ell})$ a directed path in the influence graph $g$. 
	The length of $\pi$ is 
	$\text{len}(\pi)=\ell$ and the weight of $\pi$ is given by the product of all edge weights along that 
	path: 
	\begin{equation}
		F(\pi) = \prod_{k=1}^{\ell} F(i_{k-1} \to i_k) \, .
	\end{equation}
	Let $\Pi = \{\pi_1 , \ldots , \pi_{M} \}$ be a set of paths. The weight of $\Pi$ is given by the sum 
	of all individual path weights:
	\begin{equation}
		F(\Pi) = \sum_{k=1}^M F(\pi_k) \, .
	\end{equation}	

    The input gramian $G(s)$ \eqref{eq:defGramian} is the composition of the transfer function $T$ and its
     hermitian  conjugate $T^*$. The transfer function $T$ can be interpreted as a gammoid $\Gamma=(\mathcal{L},g,Z)$, where 
	the input nodes from $\mathcal{L}$ correspond to the columns of $T$ and the output nodes from $Z$ 
	correspond to the rows of $T$.

	 There is also a \emph{transposed gammoid} \footnote{The transposed gammoid should not be confused with the notion of a dual gammoid in matroid 
	theory \cite{whitney_abstract_1935}.}.
	 \begin{equation}
		\Gamma'=(Z',g',\mathcal{L}') \, .
	\end{equation}	
	 corresponding to the hermitian conjugate  $T^*$, see Fig.~\ref{fig:Fig2}.	Here, the 
	 \emph{transposed graph} $g'$ is obtained by flipping the edges of the original graph $g$. 
	 The input ground set $Z'$ of the transposed
	 gammoid $\Gamma'$ corresponds to the output set $Z$ of $\Gamma$. Similarly, the output set 
	 $\mathcal{L}'$ of $\Gamma'$ is given by the input ground set $\mathcal{L}$ of $\Gamma$.

	As we have gammoid representations $\Gamma$	and $\Gamma'$ for $T$ and $T^*$, also the gramian has such a gammoid
	representation which we denote as $(\Gamma\circ \Gamma')$.
    To obtain $(\Gamma \circ \Gamma')$ we identify
	the outputs $Z$ of $\Gamma$ with the inputs $Z'$ of $\Gamma'$, see Fig.~\ref{fig:Fig2}(c). 
	\begin{figure}
		\centering
		\includegraphics[width=\columnwidth]{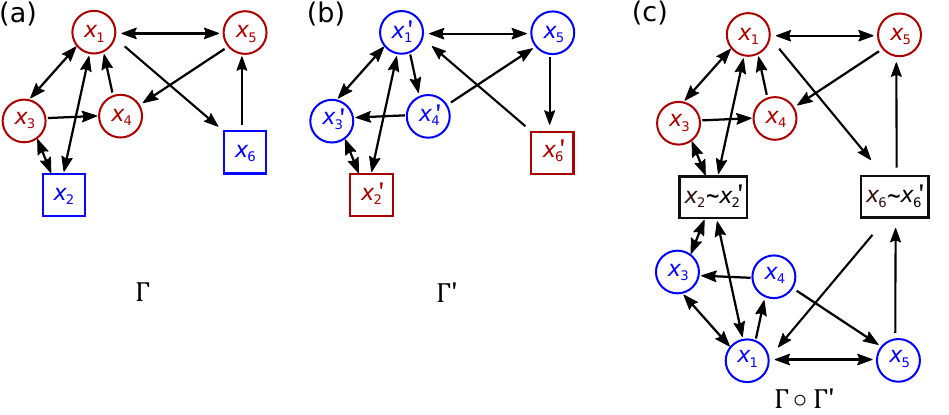}
		\caption{Gammoid representation of the input gramian. 
		(a) An exemplary gammoid $\Gamma$. The nodes in red represent the input ground set $\mathcal{L}$ and 
		the nodes in blue (squares) the output set $Z$.
		(b) The transposed gammoid $\Gamma'$. Compared to (a), the arrows are flipped. The red nodes (squares) 
		represent the input ground set $Z'$ and the nodes in blue the output set $\mathcal{L}'$.
		(c) The combined gammoid $(\Gamma \circ \Gamma')$. The outputs $Z$ of $\Gamma$ are identified with the 
		inputs $Z'$ of $\Gamma'$. Again, red nodes represent the inputs $\mathcal{L}$ and the blue nodes represent 
		the outputs $\mathcal{L}'$ of the gammoid $(\Gamma \circ \Gamma')$.
		}
		\label{fig:Fig2}
	\end{figure}

	%
	%
	\begin{definition}	 \label{def:shortestpathcoherence}
		Let $\Gamma$ be a weighted gammoid with ground set $\mathcal{L}=\{l_1,\ldots,l_L\}$. For two nodes
		$l_i,l_j\in\mathcal{L}$ let 
		$\psi_{ij}$ denote the shortest path from $l_i$ to $l_j'$ in $(\Gamma \circ \Gamma')$. We call
		\begin{equation}
			\mu_{ij}^\text{short} := \frac{\vert F(\psi_{ij})\vert}{\sqrt{F(\psi_{ii})F(\psi_{jj})}} 
		\end{equation}
		the shortest path coherence between $l_i$ and $l_j$.
	\end{definition}
	%
	%
	\begin{theorem} \label{theorem:shortestpathcoherence}
		We find that
		\begin{equation}
			\mu_{ij}^\text{short}  \geq
			\lim_{|s| \to \infty}\frac{\vert G_{ij}(s)\vert}{\sqrt{G_{ii}(s)G_ {jj}(s)}} \, .
		\end{equation}
	\end{theorem}
	We see that
	\begin{equation}
		\inf_{s\in\mathbb{C}} \max_{i\neq j} \mu_{ij}(s) \leq \max_{i \neq j} \mu_{ij}^\text{short}
	\end{equation}		
	and therefore the shortest path mutual coherence can also be used in theorem 
	\ref{theorem:sparkmu} to get a (more pessimistic) bound for the spark.
	The advantage of the shortest shortest path mutual coherence is that it can 
	readily be computed even for large ($N>100$) networks. 
\begin{figure}
		\centering
		\includegraphics[width=\textwidth]{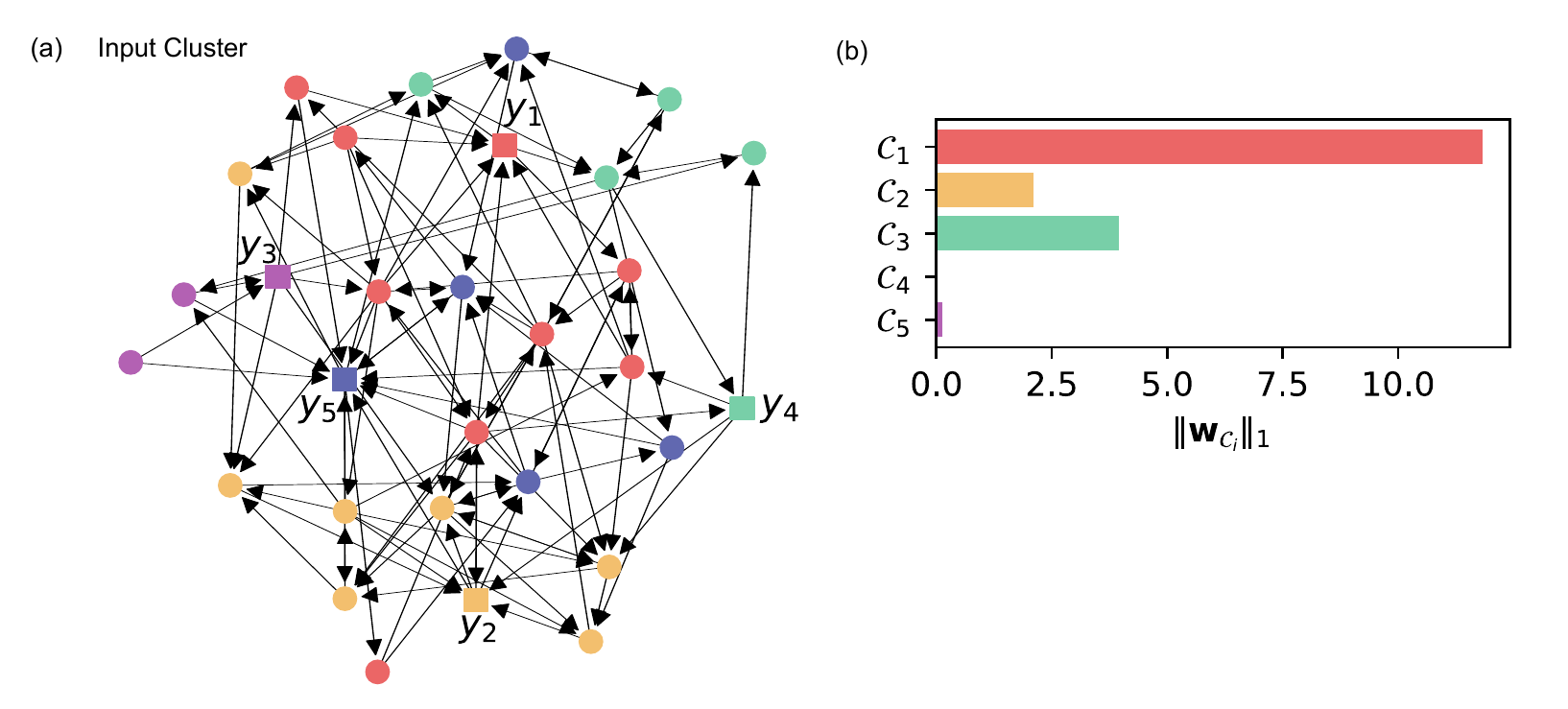}
		\caption{Restricting the error location to a subset of state nodes. 
			(a) The graph of a linear dynamic system with $30$ nodes and $5$ 
			output nodes $y_1,\ldots,y_5$. These outputs are not sufficient to 
			pinpoint the state $x_6$ as the single root cause of the unknown input 
			$\vec{w}^* (t)=(0,\ldots, 0, w_6^*(t),0,\ldots,0)^T$. 
			The 5 different colours of the state nodes indicate their membership in one of the 
			different input clusters $\mathcal{C}_1,\ldots, \mathcal{C}_5$. State nodes 
			within the same input cluster can not be distinguished from each other as 
			possible error sources. 
			(b) The sum over signal norms of each state in an input cluster (see \eqref{eq:cluster_input}) 
			can be 	used as an indicator that at least one node in this particular cluster 
			is targeted by an error.  The 
			barplot indicates that $\mathcal{C}_1$ is the cluster most likely to be targeted 
			by a 1-sparse error. }
		\label{fig:Fig3}
	\end{figure}	
\section*{Examples and Node Clustering}
In this section we illustrate by example, how our theoretical results from the previous section 
can be used to localize and reconstruct unknown inputs. These inputs can be genuine
inputs from the environment or model errors or faults in a dynamic system
\cite{engelhardt_bayesian_2017, kahl_structural_2019}.

\subsection*{Example 1: Reconstruction of the root cause of a sparse error}
We are now coming back to the scenario in Fig.\ref{fig:Fig1}. Assume, we  have detected some unexpected output behaviour in a given dynamic system. Now, we 
want to reconstruct the root cause for the detected error. 
We simulated this scenario for a linear system with $N=30$ state nodes $\mathcal{N}=\{1,\ldots,30\}$
and randomly sampled the interaction graph $g$, see again~Fig.\ref{fig:Fig1}(a). 
The outputs are given as time course measurements $y_1^{data}(t),\ldots,y_{10}^{data}(t)$ of $P=10$ randomly
selected sensor nodes $Z$, see~Fig.\ref{fig:Fig1}(b). 
In our simulation, we have added the unknown input $\vec{w}^*(t)$ 
with the only nonzero component $w^*_6(t)$ (Fig.\ref{fig:Fig1}(c)). However, we assume that 
we have no information about the localization of this unknown input. 
Thus, the ground set is $\mathcal{L}=\mathcal{N}$. 

For a network of this size, it is still possible to 
exactly compute the spark (Definition~\ref{def:spark}) of the gammoid $(\mathcal{L},g,Z)$ (Definition \ref{def:Gammoid}). This straightforward algorithm iterates 
over two different loops: In the inner loop we iterate over all possible input sets $S$ of size $r$  and check, whether $S$ is linked in $g$
into $Z$ (see Theorem~\ref{theorem:structuralinvertibility}). In the outer loop we repeat this for all possible $r=1,2,\ldots,N$. The algorithm terminates, 
if we find an input set which is not linked into $Z$. If $r$ is largest subset size for which all $S$ are linked in $g$ into $Z$, the spark is given by $r+1$.

For the network in ~Fig.\ref{fig:Fig1}(a) we find that $\text{spark}\, \Gamma=3$. From \eqref{eq:spark_localizability} we conclude, that 
an unknown input targeting a single node in the network can uniquely be localized. Thus, under the assumption that the 
output residual was caused by an error targeting a single state node, we can uniquely reconstruct this input from 
the output. The reconstruction is obtained as the  solution of the regularized optimization 
 problem in \eqref{eq:optimization},see~Fig.\ref{fig:Fig1}(c). For the fit we allowed each node $x_i$ to receive an input $\hat{w}_i$. We used a regularization constant of $\beta = 0.01$ in ~\eqref{eq:optimization}).
 
Please note, that a necessary condition for the reconstruction to work is an assumption about 
the sparsity of the unknown input. If we would assume that more than one state node is targeted by an error, we would need
a larger spark to exactly localize and reconstruct the error. This would either require a smaller ground set 
$\mathcal{L}$ or a different set of sensor nodes $Z$, or both. 

\subsection*{Example 2: Restricting the error location to a subset of state nodes} 
In Fig.~\ref{fig:Fig3} we have plotted the graph corresponding to the same system as in Fig.~\ref{fig:Fig1}, but now with
a different set of only 5 output nodes $Z$. As before, we performed a twin experiment by adding an
error signal $w^*_6(t)$ to the state node $i=6$ only. However, this smaller sensor node set $Z$ is not sufficient to localize even a single error at an unknown position. The corresponding gammoid $\Gamma=(\mathcal{L}, g, Z)$
with ground set $\mathcal{L}=\mathcal{N}$ has $\text{spark}\, \Gamma = 2$ and therefore the condition  \eqref{eq:spark_localizability} is not fulfilled. Solving the optimization problem~\eqref{eq:L1} with $\mathcal{L}$ as a ground set is also not guaranteed to reconstruct the localization of the unknown input, since the RIP condition (Theorem~\ref{theorem:RIP}) cannot practically be 
tested. 

Nevertheless, we can still restrict the position 
of the error using a node clustering strategy: We use the shortest path coherence matrix
$(\mu^\text{short}_{ij})$ as a similarity index between each pair $(i,j)\in\mathcal{N}\times \mathcal{N}$ 
of state nodes and employ a standard hierarchical clustering algorithm to group the state nodes 
in $\mathcal{L}$ into disjoint input clusters $\mathcal{C}_1,\ldots, \mathcal{C}_5$ such that $\mathcal{L}= \mathcal{C}_1 \dot{\cup} \mathcal{C}_2 \dot{\cup} \ldots \dot{\cup} \mathcal{C}_5 $. The nodes belonging to one and the same cluster 
have such a high coherence that it is not possible to determine which of them is the 
root course of the error. But, as can be seen in Fig.~\ref{fig:Fig2}(b), we can rank the input clusters 
to at least isolate the cluster containing the node targeted by the error. To this end, we 
solve the optimal control problem \eqref{eq:optimization} and sum the 
estimated hidden inputs within each cluster
\begin{equation}
\Vert \vec{w}_{\mathcal{C}_k} \Vert_1= \sum_{i \in \mathcal{C}_k}\Vert w_i \Vert_p.\label{eq:cluster_input}
\end{equation}
The strongest total input signal is observed for cluster $\mathcal{C}_1$. If we assume that the root cause 
of the observed error is localized at a single state node (sparsity $k=1$), we 
select this cluster $\mathcal{C}_1$. This corresponds to the ground truth in this twin experiment.
However, if we suspect more than one node to be targeted by errors (say $k \ge 2$) but the strongest cluster contains less than $k$ nodes, we have to consider more clusters in the order 
of their ranking until $k$ nodes are covered.

\subsection*{Example 3: Sensor node selection and a strategy for constricting the error location} 

\begin{figure*}
		\centering
		\includegraphics[width=\textwidth]{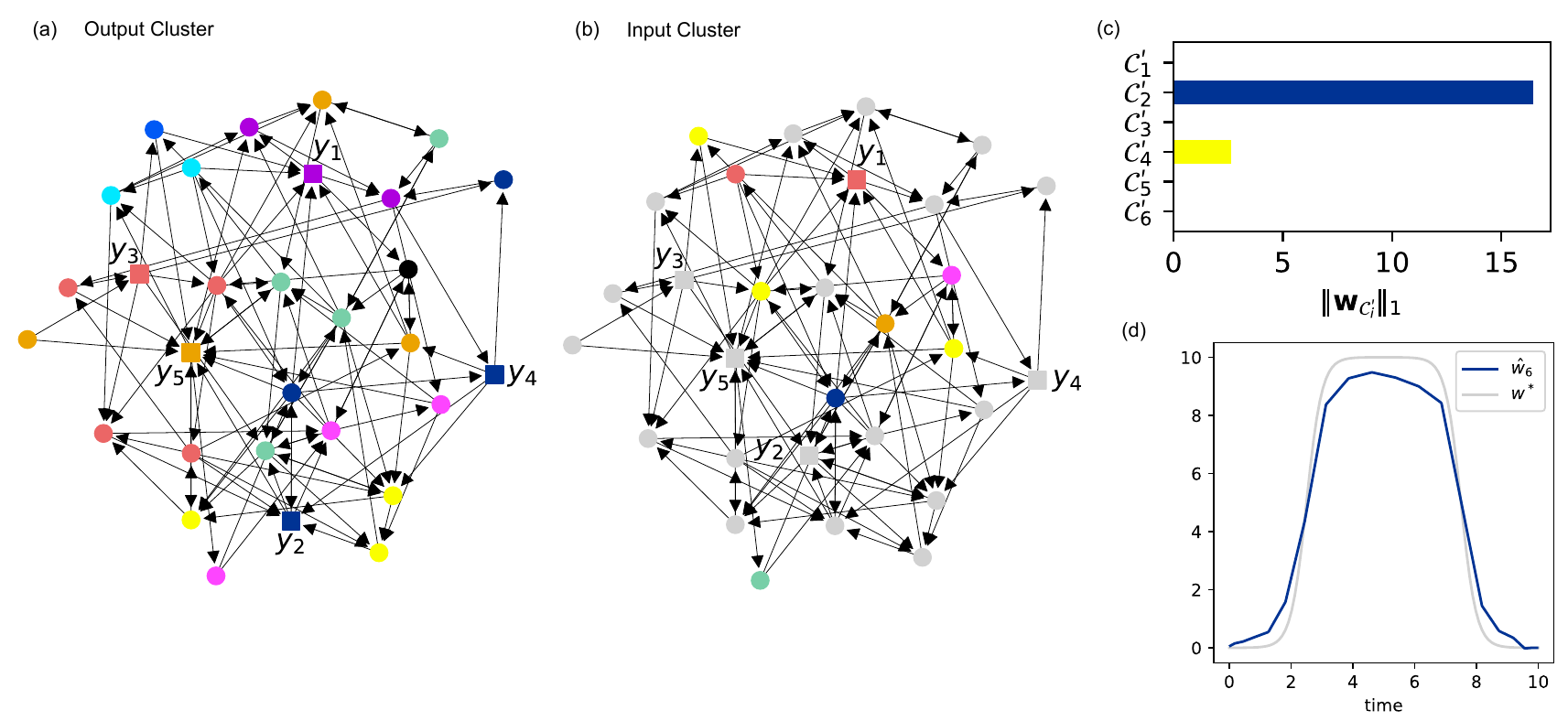}
		\caption{Narrowing down the location of the error using informative sensor nodes.  
			(a) The system, the unknown inputs and the output node set are the same as 
			in Fig.~\ref{fig:Fig3}. But, the colours in this graph indicate now clusters of 
			coherent {\em outputs}. Measuring more than one state node within the same output 
			cluster does not provide additional information about the location of the error as 
			compared to a single measurement in each cluster. 
			For example, output sensors $y_2$ and $y_4$ are in the same output cluster 
			and thus provide only redundant information. Therefore, we move the 
			sensor $y_2$ to an output cluster which is not yet covered by a sensor. 
			(b) {\em Input} 
			clustering after moving $y_2$. 
			The state nodes within each of the 6 input clusters can still not be distinguished as 
			the cause of the error. (c) Cluster $\mathcal{C}_2'$ in (b) (a subcluster of the 
			original $\mathcal{C}_1$ in Fig.~\ref{fig:Fig3})  has the largest total signal norm 
			(\eqref{eq:cluster_input}). In this example, $\mathcal{C}_2'$ consists only of the 
			single state node with $i=6$.
			(d) The reconstructed output $\hat{w}_6(t)$ is a good estimate of the true error 
			signal. 
  			}
		\label{fig:Fig4}
	\end{figure*}	

The clustering of the input ground set described in the previous Example~2 is always dependent on the given output set $Z$. 
This implies that the the size and location of the input clusters in the last example depends on the output set $Z$.
To further narrow down the position of the single unknown input in Fig.~\ref{fig:Fig2},
additional or more carefully chosen measurements are necessary. A set $Z$ of sensor nodes is most effective, 
if the sensors provide output signals with distinct information. We will describe now by example, how the 
degree of coherence between the output measurements can be quantified. This will provide a strategy for selecting incoherent, i.e. 
non-redundant, sensor nodes. 

In the same way we use $(\Gamma \circ \Gamma')$ to compute a measure for the indistinguishability of 
input nodes, we can compute the shortest path coherence matrix of $(\Gamma' \circ \Gamma)$ to quantify the 
indistinguishability of output nodes. This output coherence measure depends on the chosen input set. 
Since we have already identified $\mathcal{C}_1$ as the target cluster for the error (see Fig.~\ref{fig:Fig3}), we now 
restrict the input set to $\mathcal{C}_1$ and compute the output coherence with respect to this 
restricted ground set~$\mathcal{C}_1$. Fig.~\ref{fig:Fig4}(a) shows the resulting clusters of 
indistinguishable output nodes.  Sensor nodes in the same  output cluster provide similar, 
and thus redundant information. Hence, it makes sense to place the sensors in distinct output clusters. 

We can see in Fig.~\ref{fig:Fig4}(a) that the sensors $y_2$ and $y_4$ lie in the same output cluster. To 
collect complementary information about the location of the unknown input, we replace one of these
two sensors, say $y_2$, by a sensor in a different cluster. In Fig.~\ref{fig:Fig4}(b), the 
output $y_2$ was moved to a different position and thus covers a different output cluster.
We can now repeat the input clustering from Example 2 for the gammoid with $\mathcal{C}_1$ as input ground set 
and our improved output set. The resulting input clustering can be seen in Fig.~\ref{fig:Fig4}(b). 
The grey nodes represent the state variables which are not in the input ground set and therefore not considered in the clustering. 
With~$\mathcal{C}_1$ as input ground set we can again solve the optimization problem~\eqref{eq:optimization} 
and add the input norms (see \eqref{eq:cluster_input}) in each of the new subclusters of $\mathcal{C}_1$. 
From Fig.~\ref{fig:Fig4}(c) we identify the new cluster $\mathcal{C}_2'$ (a subcluster of the original $\mathcal{C}_1$) 
as the target cluster for unknown inputs. In this example, $\mathcal{C}_2'$ consists only of the single 
state node with $i=6$. Indeed, this is correct input node in our twin experiment.  Solving the 
optimization problem one last time with only the single node target cluster $\mathcal{C}'_2=\{6\}$ as ground set 
provides  again an accurate reconstruction $\hat{w}_6(t)$ of the error signal, see Fig.~\ref{fig:Fig4}(d). 

This example illustrates, how the input ground set can be more and more 
reduced by iteratively repeating the strategy consisting of input node clustering, target cluster selection 
using convex optimization, output clustering and non-redundant sensor node placement. 

Please note, that this strategy depends on the assumption of sparsity. Here, we assumed
a single input, i.e. a sparsity of $k=1$. However, it is straightforward 
to generalize to $k \ge 2$ sparse errors, as long as these potential target clusters 
cover at least $k$ potential input nodes. Also note, that in the case of real
data with stochastic input noise, there might be several clusters with almost 
the same total input signal. Then, it might be reasonable to also consider 
clusters with approximately the same total input. Analysing the effect of stochastic measurement
noise is left as a question for further research.

\section*{Discussion}
Finding the root cause of errors or faults is important in many contexts. We have presented 
a mathematical theory for the localization of sparse errors, which overcomes the need to 
a priori assume certain types of errors. This restriction is replaced by the sparsity assumption, which 
is plausible in many real world settings, where the failure of a small number of components is 
observed from the sensors, but the localization of the fault is unknown. Similarly, for the problem 
of modelling dynamic systems, it is important to know where the model is wrong and which states in the 
model need a modified description. This includes also open systems, which are influenced by unknown 
inputs from their environment. We have used the gammoid concept to define the notion of independence 
for inputs to dynamic systems. This allowed us to generalize concepts from sparse sensing to 
localize and infer such sparse unknown inputs.

Theorem~\ref{theorem:spark2} is general and applies to nonlinear systems. The other results are only proved for linear systems. An important open research question is to check, whether 
the optimization problem  in~\eqref{eq:L1} is also suitable for nonlinear systems. 
In addition, the RIP-condition in Definition~\ref{def:RIP} is already hard to test for linear systems, a situation we already know from classical compressed sensing for matrices~\cite{yonina_c_eldar_compressed_nodate}. However, there is an additional complication in the problem of estimating 
the inverse of the input-output map $\Phi$ corresponding to the dynamic  system \eqref{eq:DynamicSystem} : The 
map $\Phi$ is compact and maps from an infinite dimensional input space to the infinite dimensional output space. Inverse systems theory \cite{nakamura_potthast_inverse_2015} tells us, that the inversion of such operators is discontinous. Thus, more research on the numerics of this $L_1$ regularized optimal 
control problem is needed \cite{vossen_onl1-minimization_2006}. In addition, stochastic dynamic
systems with unknown inputs will provide an interesting direction for further research.

Our results are complementary to recent work on Data-Driven Dynamic Systems, where the the goal is to discover the dynamics solely from measurement data \cite{brunton_discovering_2016, yair_reconstruction_2017, pathak_model-free_2018, champion_data-driven_2019}. For data sets of limited size, these purely data driven methods might profit from the prior knowledge encoded by an possibly imperfect but informative model. Our work provides a straightforward approach to combine models and data driven methods. For a given model, the estimated error signals can be analysed with a data driven method to discover their inherent dynamics. We believe that the
combination of data driven systems with the prior information from interpretable mechanistic models will provide major advances in our understanding of dynamic networks.

\section*{Acknoledgements}
	The authors would like to thank Philipp Wendlandt and Jörg Zimmermann for the valuable discussions throughout the creation process. 
The present work is part of the SEEDS project, funded by \textit{Deutsche Forschungsgemeinschaft (DFG)}, project number 354645666.

\appendix
\section{Spaces and Norms}
	Our gammoid approach enables us to generalize 
	concepts from compressed sensing \cite{donoho_optimally_2003, candes_robust_2006} of 
	matrices to and dynamic systems. Let us start with the matrix case of classical compressed
	sensing:
	
	For a given matrix 
	$A\in\mathbb{R}^{P\times N}$, $M>>P$, and given $\vec{y}\in\mathbb{R}^P$, solve
	\begin{equation}
		A\vec{x} = \vec{y}
	\end{equation}
	for $\vec{x}$. We will refer to this as the \textit{static problem}. In contrast to that we consider the 
	problem: For a given input-output map $\Phi:\mathcal{U}\to \mathcal{Y}$ where $\mathcal{U}
	=\mathcal{U}_1\oplus  \ldots \oplus \mathcal{U}_N$ and 
	$\mathcal{Y}=\mathcal{Y}_1 \oplus \ldots \oplus \mathcal{Y}_P$, $\mathcal{U}_i,\mathcal{Y}_j$ 
	function spaces, and given 
	$\vec{y}\in\mathcal{Y}$, solve
	\begin{equation}
		\Phi(\vec{u}) = \vec{y}
	\end{equation}
	for $\vec{u}$. This will be called the \textit{dynamic problem}. In fact, until we make use of the gammoid 
	structure of a dynamic input-output system, $\mathcal{U}_i$ and $\mathcal{Y}_j$ can be arbitrary Banach 
	spaces and $\Phi$ any operator. Despite this more general validity, we will throughout call 
	$\mathcal{U}$ the input space and $\mathcal{Y}$ the output space. 
	The \textit{sparse sensing problem}, either static as considered or dynamic, is understood as finding 
	the sparsest 
	solution of the static or dynamic problem, i.e., a solution $\hat{\vec{u}}$ for which most of the components
	$\hat{u}_i$ are identically zero.
	Assume a component, say $\hat{u}_1$, has nonzero norm. 
	Then $\tilde{\vec{u}}=(0,\hat{u}_2,\ldots,\hat{u}_L)$ can be considered a compression of 
	$\hat{\vec{u}}$, since $\tilde{\vec{u}}$ is sparser than $\hat{\vec{u}}$.
	If we now find, that the compressed vector yields
	\begin{equation}
		\Vert \Phi(\tilde{\vec{u}}) - \vec{y} \Vert < \epsilon
	\end{equation}
	for some given $\epsilon$, that is, if it reproduces the desired output with sufficient 
	accuracy, we call 
	it a solution of the \textit{compressed sensing problem}.
	The static sparse and compressed sensing problem have been developed into various directions, see for 
	example \cite{yonina_c_eldar_compressed_nodate} and \cite{foucart_mathematical_2013} which both are 
	good compilations of the field and from where we borrow some notations.
	A crucial step towards a dynamic compressed sensing will be to 
	define appropriate norms on the spaces $\mathcal{U}$ and $\mathcal{Y}$.

\subsection*{Input Space}
	Henceforth we will assume the input component $u_i\in\mathcal{U}_i$ to lie in $L^p([0,T])$ for a fixed 
	$p$ and be piecewise $C^\infty[0,T]$.
\subsubsection*{Underline Notation}
	We first introduce the underline notation
	\begin{equation}
		\underline{u_i} := \Vert u_i \Vert_{L^p}
	\end{equation}
	and for the whole vector
	$\vec{u}\in \mathcal{U}$ we write
	\begin{equation}
		\underline{\vec{u}} := \begin{pmatrix}
			\underline{u_1} \\ \vdots \\ \underline{u_L}
		\end{pmatrix} \, ,
	\end{equation}
	so the underline operator maps $\vec{u}$ to a vector $\underline{\vec{u}}\in\mathbb{R}^N$. In the more 
	general setting, $\mathcal{U}_i$ is only assumed to be a Banach space with some norm $\Vert \cdot 
	\Vert_{\mathcal{U}_i}$ and the underline maps according to this norm. 
	The underline notation will help clarify our understanding of sparsity and 
	will be useful for the formulation of theorems and proofs.
	As the underline is basically a norm, it inherits the properties of a norm.	
	The following lemma merely represent calculation rules for underlined vectors.
	%
	%
	\begin{lemma} \label{lemma:underlinezero}
		\begin{enumerate}
			\item For $u_i\in\mathcal{U}_i$ we find 
			$\underline{u_i} =0 \,  \Leftrightarrow \, u_i(t) =0 \text{ a.e. in } [0,T] $ .
			\item For $u_i,v_i\in\mathcal{U}_i$ we find 
			$\underline{u_i+v_i} \leq \underline{u_i} + \underline{v_i}$.
			\item For $u_i\in\mathcal{U}_i$ and $a\in\mathbb{R}$ we find $\underline{a u_i}=\vert a\vert 
			\underline{u_i}$.
		\end{enumerate}				
	\end{lemma}
	%
	%
	\begin{proof}
		\begin{enumerate}
		\item
		From the definition of the $L^p$ spaces we get
		\begin{equation}
			\Vert u_i \Vert_{L^p} = 0 \, \Leftrightarrow \, u_i = 0 \text{ a.e.}
		\end{equation}
		\item 
		Again from the definition we find
		\begin{equation}
			\underline{u_i + v_i} = \Vert u_i +v_i \Vert_{L^p} \leq 
			\Vert u_i \Vert_{L^p} + \Vert v_i \Vert_{L^p} = \underline{u_i} + \underline{v_i} \, .
		\end{equation}
		\item Since the $L^p$-norm is homogeneous we have
		\begin{equation}
			\underline{a u_i} = \Vert a u_i \Vert_{L^p} = \vert a\vert \Vert u_i \Vert_{L^p} = 
			\vert a\vert \underline{u_i} \, .
		\end{equation}
		\end{enumerate}
	\end{proof}
\subsubsection*{Proper Norm for Inputs}
	Utilizing the underline notation, we are ready to define the $q$-norm on $\mathcal{U}$ as
	\begin{equation}
		\Vert \vec{u} \Vert_q := \Vert \underline{\vec{u}} \Vert_q
	\end{equation}		
	where on the right hand side the standard $q$-norm on $\mathbb{R}^L$ is understood.
	%
	%
	\begin{proposition}	\label{prop:properNorm}
		The $q$-norm on $\mathcal{U}$ is a proper norm.
	\end{proposition}
	%
	%
	\begin{proof}
		The $L^p$ norm and the $q$-norm on 
		$\mathbb{R}^N$ both are proper norms that fulfil the three properties 
		\textit{positive definiteness}, 
		the \textit{triangle property}, and \textit{homogeneity}. 
		Since the $q$-norm on $\mathcal{U}$ is a combination of these two norms, it becomes clear, that it again 
		fulfils these properties.
		\begin{enumerate}
			\item Positive definiteness. We find
			\begin{equation}
				\Vert \vec{u} \Vert_q = \Vert \underline{\vec{u}} \Vert_q \geq 0
			\end{equation}
			where the inequality comes from the fact, that we have a proper norm on $\mathbb{R}^L$. 
			Equality holds 
			if and only if $\underline{\vec{u}}=0$. Due to lemma \ref{lemma:underlinezero} this is the case 
			if and only if $\vec{u}=0$.
			\item Triangle inequality. Let $\vec{u},\vec{v}\in \mathcal{U}$. 
			\begin{equation}
				\Vert \vec{u} + \vec{v} \Vert_q = 
				\Vert \underline{\vec{u} +\vec{v}} \Vert_q =\left( 
				\sum_{i=1}^N (\underline{u_i+v_i} )^q \right)^{1/q}
			\end{equation}
			From lemma \ref{lemma:underlinezero} one can see, that
			\begin{equation}
				(\underline{u_i + v_i})^q \leq (\underline{u_i}+\underline{v_i})^q
			\end{equation}
			so we get
			\begin{equation}
				\Vert \vec{u} + \vec{v} \Vert_q \leq \left( 
				\sum_{i=1}^N (\underline{u_i}+\underline{v_i} )^q \right)^{1/q} = 
				\Vert \underline{\vec{u}} + \underline{\vec{v}} \Vert_q
			\end{equation}
			and since the latter is a proper $q$-norm on $\mathbb{R}^L$
			\begin{equation}
			\Vert \vec{u} + \vec{v} \Vert_q \leq \Vert \underline{\vec{u}} \Vert_q + \Vert 
			\underline{\vec{v}} \Vert_q = \Vert \vec{u} \Vert_q + \Vert \vec{v} \Vert_q \, .
			\end{equation}
			\item 
			Finally we proof homogeneity. Let $\vec{u}\in \mathcal{U}$ and $a\in\mathbb{R}$.
			\begin{equation}
				\Vert a \vec{u} \Vert_q = \Vert \underline{a \vec{u}} \Vert_q 
				= \left( \sum_{i=1}^L ( \underline{a u_i})^q \right)^{1/q}
				=\left( \sum_{i=1}^L ( \vert a \vert \underline{ u_i})^q \right)^{1/q}
				= \vert a \vert  \left( \sum_{i=1}^L ( \underline{ u_i})^q \right)^{1/q}
				= \vert a \vert \Vert \vec{u} \Vert_q
			\end{equation}
		\end{enumerate}
	\end{proof}		
	For $q=0$ one derives a situation comparable to the ``$0$-norm'' on $\mathbb{R}^L$. First, 
	\begin{equation}
		\Vert \vec{u} \Vert_0 := \Vert \underline{\vec{u}} \Vert_0
	\end{equation}
	counts the non-zero components of $\underline{\vec{u}}$. Due to lemma \ref{lemma:underlinezero}, a 
	component of $\underline{\vec{u}}$ is zero if and only if the corresponding component of $\vec{u}$ is zero, 
	so $\Vert \vec{u} \Vert_0$ is indeed a ``$0$-norm'' on $\mathcal{U}$.
\subsubsection*{Support of an Input}
	Similar to \cite{yonina_c_eldar_compressed_nodate}, for an index set 
	$\Lambda \subseteq \{1,\ldots , N\}$ we write $\vec{u}_\Lambda$ for the 
	vector
	\begin{equation}
		(\vec{u}_\Lambda)_i = \left\{ \begin{aligned} u_i &\text{ if }i\in \Lambda \\
		0 &\text{ if }i\not\in \Lambda  \end{aligned} \right. 
	\end{equation}
	and with $\Lambda^c$ we denote the complement in $\{1,\ldots,N\}$. 
	As in \cite{foucart_mathematical_2013} we call $\Lambda$ the \emph{support} of $\vec{u}$, if $\Lambda$ 
	is of minimal cardinality and $\vec{u}_\Lambda = 
	\vec{u}$. 
	Let the index set $\Lambda$ be given, then
	\begin{equation}
		\mathcal{U}_\Lambda := \{\vec{u}\in\mathcal{U} \, | \, \text{supp}\,\vec{u} \subseteq \Lambda \}
	\end{equation}
	is understood as a restriction of the input space. 
	
	If the operator $\Phi$ is linear, then it can be written as a 
	matrix with components $\Phi_{ji}:\mathcal{U}_i\to\mathcal{Y}_j$. The restriction of the domain 
	\begin{equation}
		\Phi : \mathcal{U}_\Lambda \to \mathcal{Y}
	\end{equation}
	is then equivalent to deleting the columns of $\Phi$ with index not included in $\Lambda=(\lambda_1,
	\ldots,\lambda_M)$. We can also interpret it as a new operator $\Phi_{\Lambda}$ with unrestricted domain,
	\begin{equation}
		\Phi_\Lambda : \mathcal{U}_{\lambda_1} \oplus \ldots \mathcal{U}_{\lambda_M} \to \mathcal{Y} \, .
	\end{equation}
	If $\mathcal{U}=\mathbb{R}^L$ and $\mathcal{Y}=\mathbb{R}^P$, then $\Phi_\Lambda$ is simply a 
	$\mathbb{R}^{P\times M}$ matrix. This case coincides with the static problem and was first 
	considered in \cite{donoho_compressed_2006}. We will make use of the following 
	notation from \cite{yonina_c_eldar_compressed_nodate}. 
	
	The ``$0$-norm'' can also be expressed via the support,
	\begin{equation}
		\text{card}\,\text{supp}\, \vec{u}= \Vert \vec{u} \Vert_0 \, .
	\end{equation}
	The set of $k$-sparse inputs is understood as 
	\begin{equation}
		\Sigma_k := \{ \vec{u}\in \mathcal{U} \,|\,\Vert \vec{u} \Vert_0 \leq k   \} \, .
	\end{equation}
	One will see, that $\Sigma_k$ is the union of all $\mathcal{U}_\Lambda$ with 
	$\text{card}\, \Lambda \leq k$.

	The three following facts can be proven directly:
	Let $\Lambda_0$ and $\Lambda_1$ be two index sets with 
	$\text{card}\, \Lambda_0 = \text{card}\, \Lambda_1 =k$ and let $\vec{u}\in \mathcal{U}$. We find 
	\begin{equation}
		\vec{u}_{\Lambda_0},\vec{u}_{\Lambda_1} \in \Sigma_k
	\end{equation}
	and 
	\begin{equation}
		\vec{u}_{\Lambda_0}+\vec{u}_{\Lambda_1} \in \Sigma_{2k} \, .
	\end{equation}
	If $\Lambda_0$ and $\Lambda_1$ are disjoint we also get
	\begin{equation}
		\vec{u}_{\Lambda_0} + \vec{u}_{\Lambda_1} = \vec{u}_{\Lambda_0 \cup \Lambda_1} \, . 
		\label{eq:LambdaDisjoint}
	\end{equation}
	%
	%
	\begin{lemma} \label{lemma:disjointSupport}
		Let $\vec{u},\vec{v}\in\mathcal{U}$ with disjoint support, then 
		\begin{equation}
			\underline{\vec{u} \pm \vec{v}} = \underline{ \vec{u} } + \underline{\vec{v}} 
			\label{eq:l1}
		\end{equation}
		and for the $q$-norm 
		\begin{equation}
			\Vert \underline{\vec{u}} - \underline{\vec{v}} \Vert_q = \Vert 
			\vec{u}\pm \vec{v} \Vert_q = \Vert \underline{\vec{u}} + \underline{\vec{v}} \Vert_q
		\end{equation}
	\end{lemma}
	%
	%
	\begin{proof}
		If $i$ is in the support of $\vec{u}$, then $v_i=0$ and equation \eqref{eq:l1} reduces to
		\begin{equation}
			\underline{u_i} = \underline{u_i} 
		\end{equation}
		which is true.
		If $i$ is in the support of $\vec{v}$, then $u_i=0$. Equation \eqref{eq:l1} together with lemma 
		\ref{lemma:underlinezero} then becomes
		\begin{equation}
			\underline{\pm v_i}= \vert \pm 1 \vert \underline{v_i} = \underline{v_i}\label{eq:l2}
		\end{equation}
		which also holds true. 
		If $i$ is in neither support, the equation becomes trivial.
		
		For equation \eqref{eq:l2} note, that
		\begin{equation}
			\Vert \vec{u} \pm \vec{v} \Vert_q = \Vert \underline{\vec{u} \pm \vec{v}} \Vert_q =
			\Vert \underline{\vec{u}} + \underline{\vec{v}} \Vert_q
		\end{equation}
		where the first equality is clear by definition and the second equality was just proven.
				
		It remains to show, that for any two vectors $\vec{x},\vec{y}\in\mathbb{R}^N$ with disjoint support 
		one gets
		\begin{equation}
			\Vert \vec{x} - \vec{y} \Vert_q = \Vert \vec{x} + \vec{y} \Vert_q \, .
		\end{equation}
		However due to the disjoint support $\vert x_i +y_i \vert^q = \vert x_i\vert^q + \vert y_i \vert^q =
		\vert x_i-y_i \vert^q$ since at least one of the two terms is zero. One finds
		\begin{equation}
			\Vert \vec{x} - \vec{y} \Vert_q^q =\sum_{i=1}^N \vert x_i-y_1 \vert^q =
			\sum_{i=1}^N \vert x_i+y_1 \vert^q = \Vert \vec{x} + \vec{y} \Vert_q^q \, . 
		\end{equation}
	\end{proof}
	The following proposition for the case, where $\vec{u},\vec{v}$ are $\mathbb{R}^N$ vectors, stems from 
	\cite{bergh_interpolation_1976} and has already been used for the static problem in
	\cite{donoho_compressed_2006}. We proof its validity for $\vec{u},\vec{v}$ being input vectors of a 
	dynamic system.
	%
	%
	\begin{proposition} \label{prop:disjointSupport}
		Let $\vec{u},\vec{v}\in\mathcal{U}$ with disjoint support, then 
		\begin{equation}
			\Vert \vec{u} + \vec{v} \Vert_q^q = \Vert \vec{u} \Vert_q^q + \Vert \vec{v} \Vert_q^q \, . 
		\end{equation}
	\end{proposition}
	%
	%
	\begin{proof}
		With lemma \ref{lemma:disjointSupport} we find
		\begin{equation}
			\Vert \vec{u} + \vec{v} \Vert_q^q = \Vert \underline{ \vec{u} + \vec{v}} \Vert_q^q 
			= \Vert \underline{\vec{u}} + \underline{\vec{v}} \Vert_q^q = \sum_{i=1}^N (\underline{u_i} 
			+\underline{v_i})^q \, .
		\end{equation}
		Again due to the disjoint support we can write
		\begin{equation}
			(\underline{u_i}+\underline{v_i})^q = (\underline{u_i})^q +(\underline{v_i})^q
		\end{equation}
		to get
		\begin{equation}
			\Vert \vec{u} + \vec{v} \Vert_q^q =   \sum_{i=1}^N (\underline{u_i} )^q
			+\sum_{i=1}^N (\underline{v_i})^q  = \Vert \vec{u} \Vert_q^q + \Vert \vec{v} \Vert_q^q \, .
		\end{equation}
	\end{proof}
	We close the investigation of the input space with three lemmas on norm-inequalities. 
	For $\vec{u}$ being a $\mathbb{R}^N$ vector, these lemmas are proven in
	\cite{yonina_c_eldar_compressed_nodate}. For our purposes, it is necessary to prove the validity for 
	the $q$-norms on composite Banach spaces.
	%
	%
	\begin{lemma}\label{lemma:sqrtk}
		For $\vec{u}\in \Sigma_k$ we find
		\begin{equation}
			\frac{1}{\sqrt{k}} \Vert \vec{u} \Vert_1 \leq \Vert \vec{u} \Vert_2 \leq \sqrt{k} \Vert \vec{u} 
			\Vert_\infty \, .
		\end{equation}
	\end{lemma}
	%
	%
	\begin{proof}
		Let $\langle \cdot , \cdot \rangle$ denote the standard scalar product on $\mathbb{R}^N$. We 
		can write
		\begin{equation}
			\Vert \vec{u} \Vert_1 = \Vert \underline{ \vec{u}} \Vert_1 = \langle \underline{u} ,
			\text{sgn}\, \underline{\vec{u}}\rangle \leq \Vert \underline{\vec{u}} \Vert_2 \Vert \text{sgn}\, 
			\underline{\vec{u}} \Vert_2 \, ,
		\end{equation}
		where the latter inequality comes from Cauchy-Schwartz and the signum function is understood 
		component-wise. By assumption $\vec{u}$ is $k$-sparse, hence $\Vert \text{sgn}\, \underline{\vec{u}}
		\Vert^2_2$ 
		is a sum of  at most $k$ ones. We obtain
		\begin{equation}
			\Vert \vec{u} \Vert_1 \leq \sqrt{k} \Vert \underline{\vec{u}} \Vert_2  = 
			\sqrt{k} \Vert \vec{u} \Vert_2 \, .
		\end{equation}
		One will see, that
		\begin{equation}
			\underline{u_i} \leq \Vert \underline{\vec{u}} \Vert_\infty = \Vert \vec{u} \Vert_\infty \, ,
		\end{equation}
		where both sides of the inequality are non-negative.
		Let $\Lambda = \text{supp}\, \vec{u}$, then
		\begin{equation}
			\Vert \vec{u} \Vert_2^2 = \Vert \underline{\vec{u}} \Vert_2^2 = 
			\sum_{i\in \Lambda} (\underline{u_i})^2 \leq \Vert \vec{u} \Vert_\infty^2 \sum_{i\in \Lambda} 1
			= k \Vert \vec{u} \Vert_\infty^2 \, .
		\end{equation}
		Taking the square-root leads to the desired inequality.
	\end{proof}
	%
	%
	%
	\begin{lemma} \label{lemma:sqrt2}
		For $\vec{u},\vec{v}\in\mathcal{U}$ with disjoint support we find
		\begin{equation}
			\Vert \vec{u} \Vert_2 + \Vert \vec{v} \Vert_2 \leq \sqrt{2} \Vert \vec{u} + \vec{v} \Vert_2 \, .
		\end{equation}
	\end{lemma}
	%
	%
	\begin{proof}
		Consider the $\mathbb{R}^2$ vector
		\begin{equation}
			\vec{x} := \begin{pmatrix}
				\Vert \vec{u} \Vert_2 \\ \Vert \vec{v} \Vert_2 
			\end{pmatrix} \, .
		\end{equation}
		Lemma \ref{lemma:sqrtk} holds also for constant vectors and by construction $\vec{x}$ is 
		$k=2$ sparse, thus
		\begin{equation}
			\Vert \vec{x} \Vert_1 \leq \sqrt{2} \Vert \vec{x} \Vert_2 \, .
		\end{equation}
		We can now replace 
		\begin{equation}
			\Vert \vec{x} \Vert_1 = \Vert \vec{u} \Vert_2 + \Vert \vec{v} \Vert_2 \, .
		\end{equation}
		For the right hand side we find
		\begin{equation}
			\Vert \vec{x} \Vert_2^2 = \Vert \vec{u} \Vert_2^2 + \Vert \vec{v} \Vert_2^2 = 
			 \Vert \vec{u} + \vec{v} \Vert_2^2 
		\end{equation}
		where the second equality comes from proposition \ref{prop:disjointSupport}. Combining the latter the 
		equations leads to the desired inequality.
	\end{proof}
	%
	%
	%
	\begin{lemma} \label{lemma:LambdaSums}
		Let $\vec{u}\in\mathcal{U}$ an $\Lambda_0$ and index set of cardinality $k$. For better 
		readability we write $\vec{x}:=\underline{\vec{u}_{\Lambda_0^c}}$. Note, that $\vec{x}$ is a 
		$\mathbb{R}^N$ vector with non-negative components. Let $L = (l_1,\ldots, l_N)$ 
		a list of indices with $l_i \neq l_j$ for $i\neq j$ such that for the components of $\vec{x}$ we find
		\begin{equation}
			x_{l_1} \geq x_{l_2} \geq \ldots \geq x_{l_N} \, .
		\end{equation}
		Define index sets $\Lambda_1:=\{l_1,\ldots,l_k \}$, $\Lambda_2:=\{l_{k+1},\ldots ,
		l_{2k} \}$ and so forth until the whole list $L$ is covered. If $k$ is not a divisor of $N$ the last 
		index set would have less than $k$ elements. This can be fixed by appending $\vec{u}$ by zero elements.
		
		We find
		\begin{equation}
			\sum_{j\geq 2} \Vert \vec{u}_{\Lambda_j} \Vert_2 \leq  \frac{\Vert  \vec{u}_{\Lambda_0^c} 
			\Vert_1}{\sqrt{k}}
		\end{equation}
	\end{lemma}
	%
	%
	\begin{proof}
		First note, that by construction all $\Lambda_i$ for $i=0,1,\ldots$ are pair-wise disjoint and that
		\begin{equation}
			\Lambda_0^c = \Lambda_1 \dot{\cup} \Lambda_2 \dot{\cup} \ldots
		\end{equation}
		For any $m\in \Lambda_{j-1}$ we get by construction
		\begin{equation}
			\underline{u_m} \geq \Vert \vec{u}_{\Lambda_j} \Vert_\infty \, . 
		\end{equation}
		We now take the sum over all $m\in \Lambda_{j-1}$
		\begin{equation}
			\Vert \vec{u}_{\Lambda_{j-1}} \Vert_1 \geq k \Vert \vec{u}_{\Lambda_j}\Vert_\infty \, .
		\end{equation}
		From lemma \ref{lemma:sqrtk} we have for each $j$
		\begin{equation}
			\Vert \vec{u}_{\Lambda_j} \Vert_2 \leq \sqrt{k} \Vert \vec{u}_{\Lambda_j} \Vert_\infty
		\end{equation}
		and in combination with the latter inequality 
		\begin{equation}
			\Vert \vec{u}_{\Lambda_j} \Vert_2 \leq \frac{1}{\sqrt{k}} \Vert \vec{u}_{\Lambda_{j-1}} \Vert_1 \, .
		\end{equation}
		We take the sum over $j\geq 2$
		\begin{equation}
			\sum_{j\geq 2} \Vert \vec{u}_{\Lambda_j} \Vert_2 \leq 
			\sum_{j\geq 2} \frac{1}{\sqrt{k}} \Vert \vec{u}_{\Lambda_{j-1}} \Vert_1 \, .
		\end{equation}
		In the right hand side we first perform an index shift and then use proposition 
		\ref{prop:disjointSupport} to see that
		\begin{equation}
			\sum_{j\geq 2} \Vert \vec{u}_{\Lambda_{j-1}} \Vert_1 = \sum_{j\geq 1} \Vert \vec{u}_{\Lambda_j}
			\Vert_1 = \Vert \vec{u}_{\Lambda_1\cup \Lambda_2\cup \ldots} \Vert_1 = \Vert \vec{u}_{\Lambda_0^c}
			\Vert_1  \, . 
		\end{equation}
	\end{proof}
	
\subsection*{Output Space}
\subsubsection*{Underline Notation}
	In analogy to the input space $\mathcal{U}$, the underline notation can be used for the output 
	space $\mathcal{Y}=\mathcal{Y}_1 oplus \ldots \oplus \mathcal{Y}_P$,
	\begin{equation}
		\underline{y_i} := \Vert y_i \Vert_{p'} \, .
	\end{equation}
	Here, we assume that each $y_i\in \mathcal{Y}_i$ is in $L^{p'}[0,T]$ an piecewise $C^\infty[0,T]$. 
	Principally, $p'$ does not have to match the parameter $p$ from the input spaces. It will be suppressed 
	from out notation as the result are valid for any fixed value of $p'$. Again it would also be 
	sufficient that $\mathcal{Y}_i$ is a Banach space. 
	The $q$-norm on $\mathcal{Y}$ is defined
	\begin{equation}
		\Vert \vec{y} \Vert_q := \Vert \underline{\vec{y}} \Vert_q \, .
	\end{equation}
	Clearly, all rules we have derived for underlined input vectors hold true for underlined output 
	vectors.

	The following lemma yields a last inequality for underlined vectors. 
	%
	%
	\begin{lemma} \label{lemma:2norminequality}
		Let $\vec{y},\vec{z}\in\mathcal{Y}$, then
		\begin{equation}
			\Vert \underline{\vec{y}} - \underline{\vec{z}} \Vert_2^2 \leq \Vert \vec{y} + \vec{z} \Vert_2^2 \, .
		\end{equation}
	\end{lemma}
	%
	%
	\begin{proof}
		The inequality can be written as
		\begin{equation}
			\sum_{i=1}^P\left(\underline{y_i} - \underline{z_i} \right)^2 \leq 
			\sum_{i=1}^{P} \left( \underline{y_i+z_i}  \right)^2 \, .
	\end{equation}	
	To prove the validity of the latter inequality it suffices to show that
	\begin{equation}
		\vert \underline{y_i} - \underline{z_i} \vert \leq 
			\vert \underline{y_i+z_i} \vert \label{eq:yzinequality}
	\end{equation}	
	for each $i$ in order to complete the proof. 
	First, consider the case $\underline{y_i} \geq \underline{z_i}$. From lemma \ref{lemma:underlinezero} we 
	get
	\begin{equation}
		\underline{y_i} = \underline{ (y_i + z_i) - z_i} \leq  \underline{y_i + z_i} + \underline{z_i}
	\end{equation}
	and subtracting $\underline{z_i}$ on both sides yields
	\begin{equation}
		\underline{y_i} - \underline{z_i} \leq \underline{y_i+z_i} \, .
	\end{equation}
	Both sides are positive so \eqref{eq:yzinequality} holds. 
	Second, consider the case 
	$\underline{z_i} \geq \underline{y_i}$ and perform the same steps with $z_i$ and $y_i$ swapped to get
	\begin{equation}
		\underline{z_i} - \underline{y_i} \leq \underline{z_i+y_i} \, . 
	\end{equation}
	For the right hand side it is clear that
	\begin{equation}
		 \underline{z_i+y_i} =  \underline{y_i+z_i}
	\end{equation}
	and for the left hand side
	\begin{equation}
		\underline{z_i} - \underline{y_i} = \vert 	\underline{y_i} - \underline{z_i} \vert
	\end{equation}
	thus equation \eqref{eq:yzinequality} holds also in this case. 
	\end{proof}		

\section{A Note on the Gammoid of a Dynamic Input-Output System}
\subsection*{Input Sets}
	The name invertibility can be traced back to the historical development of the 
	subject. Mathematically, a system is invertible if the input-output map
	\begin{equation}
		\Phi: \mathcal{U} \to \mathcal{Y} \label{eq:PhiNoRestrictions}
	\end{equation}
	is one-to-one.
	
	We have already introduced the restricted problem
	\begin{equation}
		\Phi: \mathcal{U}_S \to \mathcal{Y} \label{eq:PhiRestrictions}
	\end{equation}
	where we only allow input vectors $\vec{u}$ whose support lies in $S\subseteq \{ 1, \ldots, N \}$.

	For the graphical interpretation, we understand $S$ as the \emph{input set}.
	
	For a given input set $S$, one will now see that the operator $\Phi$ might be one-to-one for the 
	restricted problem \eqref{eq:PhiRestrictions} while for the unrestricted problem 
	\eqref{eq:PhiNoRestrictions} it is not. 
	
\subsection*{Influence Graph and Structural Invertibility}
	The definition of the influence graph $g=(\mathcal{N},\mathcal{E})$ of a dynamic system is given in 
	the main text.
	By construction, an input set $S\subseteq \mathcal{N}$ can be interpreted as a set of nodes in the 
	influence graph. Such a node is also called an \emph{input node}.
	
	The observables of an input-output system are given by a map $\vec{c}:\vec{x}\mapsto \vec{y}$,
	\begin{equation}
		\vec{y}(t) = \vec{c}(\vec{x}(t)) \, .
	\end{equation}
	If, for instance, $c_1(\vec{x})= x_k$, then $y_1 = x_k$ and we can interpret
	the node $k\in\mathcal{N}$ as \emph{output node}. 
	Indeed, it is without loss of generality in both, the linear and the non-linear case, that each output 
	component corresponds to one state component.
	Hence, we get an \emph{out set} $Z=\{z_1,\ldots,z_P\}\subseteq \mathcal{N}$ that characterizes the 
	observables of a system from the graph theoretical point of view.
	
\subsubsection*{Sufficiency of Structural Invertibility}

	The structural invertibility of a dynamic input-output system is 
	now completely determined by the triplet $(S,g,Z)$, as the system is structurally invertible if and 
	only if $S$ is linked in $g$ into $Z$.
	It has been shown \cite{wey_rank_1998}, that structural invertibility of the influence graph 
	is necessary for the invertibility of a system.
	
	To the best of our knowledge, the sufficiency of structural invertibility for the invertibility of a 
	system is an open question.
	However, we have found the intermediate result:
	
	Say we have an dynamic system with input-output map $\Phi:\mathcal{U}_S \to \mathcal{Y}$
	where the outputs are characterized by the output set
	\begin{equation}
		Z = (z_1,\ldots,z_P) \, .
	\end{equation}
	Assume this system is invertible. Necessarily, the triplet $(S,g,Z)$ is structurally invertible, 
	where $g=(\mathcal{N},\mathcal{E})$ is the influence graph of the system. 
	Now assume, there is a distinct set
	\begin{equation}
		\tilde{Z} = (\tilde{z}_1,\ldots,\tilde{z}_P)
	\end{equation}
	with either $\tilde{z}_i = z_i$ or $(z_i\to \tilde{z}_i)\in\mathcal{E}$. 
	Note, that the output signals from $Z$ carry enough information to infer the unknown inputs of the 
	system. 		
	As each $\tilde{z}_i$ is either equal or directly influenced by $z_i$, it is plausible, that 
	this information is directly passed from $Z$ to $\tilde{Z}$. Unless, we encounter a pathological 
	situation where information is lost in this step. 
	
	To give an example for such a pathological situation is given for instance in the toy system
	\begin{equation}
		\begin{aligned}
		\dot{x}_1(t) &= x_2(t) \\
		\dot{x}_2(t) &= -x_1(t) \\
		\dot{x}_3(t) &= x_1(t) x_2(t) \\
		\dot{x}_4(t) &= \frac{x_2^2(t) - x_1^2(t)}{ x_3(t) } \, .
		\end{aligned} 
	\end{equation}
	We find the differential equation
	\begin{equation}
		\frac{\text{d}}{\text{d}t} f + x_3 g = 0 
	\end{equation}
	is solved by the right hand sides $f= x_1 x_2$ and $g={x_2^2 - x_1^2}/{ x_3 }$.
	Say the functions $x_1$ and $x_2$ yield \textit{independent information}. The information content from 
	$x_3$ and $x_4$ is not independent any more, but coupled through the differential equation above. 
		
\subsection*{Gammoid Structure}
	Consider the triplet $(S,g,Z)$ with graph $g=(\mathcal{N},\mathcal{E})$ and input and output sets 
	$S,Z\subseteq \mathcal{N}$. 
	The size of $S$ and $Z$ be $\text{card}\, S = M$ and $\text{card}\, Z = P$. 
	For a path 
	\begin{equation} 
	\pi=(n_0\to\ldots \to n_k) 
	\end{equation} 
	with $n_i\in\mathcal{N}$ and $(n_i \to n_{i+1})\in\mathcal{E}$ we call $\text{ini}\, (\pi)=n_0$ 
	the initial node and $\text{ter}\, (\pi)=n_k$ the terminal node. 
	For a set $\Pi$ of paths $\text{ini}\,\Pi$ and $\text{ter}\, \Pi$ yield the sets of initial and 
	terminal nodes.
	
\begin{definition}
	We say \emph{$S$ is linked in 
	$g$ to $Z$}, if there is a family $\Pi=(\pi_1,\ldots,\pi_M)$ of pairwise node-disjoint paths with 
	$\text{ini}(\Pi) = S$ and $\text{ter}\, (\Pi)=Z$. 
	
	The set \emph{$S$ is linked in $g$ into $Z$}, if 
	$\text{ini}\, (\Pi)=S$ and $\text{ter}\, (\Pi)\subseteq Z$. 	
\end{definition}

	The gammoid of a dynamic system assembles all triples $(S,g,Z)$ into one structure. 

	First, consider the union
	of all allowed input nodes $\mathcal{L}$. This set is called the 
	\emph{input ground set} and the input sets are understood as the subsets $S\subseteq \mathcal{L}$.	
	Without any prior knowledge,
	$\mathcal{L}=\mathcal{N}$ is often appropriate.  
	
	Second, let $\mathcal{M}$ be the \emph{output ground set}, i.e., the set of all allowed output nodes. 
	The output set $Z\subseteq 
	\mathcal{M}$ is understood. Usually, the observables of a dynamic system are given by
	\begin{equation}
		\vec{y}(t) = \vec{c}(\vec{x}(t))
	\end{equation}
	so the function $\vec{c}$ already determines the output set $Z$. 
	In the main text we therefore focus on the case $Z = \mathcal{M}$.
	
	However, if we search for an optimal output set 
	it makes sense to first define the output nodes $\mathcal{M}$ that are principally allowed, and 
	then, e.g., minimize $Z\subseteq \mathcal{M}$ under the restriction that the structural 
	invertibility condition is fulfilled.
	
\begin{definition}
	We call $\Gamma:=(\mathcal{L},g,\mathcal{M})$ a \emph{gammoid}.
\end{definition}

	Gammoids emerge from the graph theoretical problem of node-disjoint path and have been studied earlier
	without connection to dynamic systems. In \cite{perfect_applications_1968} a
	connection was made between gammoids and independence structures. The probably first usage of the 
	word gammoid stems from 
	\cite{pym_linking_1969}. The topic has mainly been developed in
	\cite{pym_proof_1969, perfect_independence_1969, ingleton_gammoids_1973, mason_class_2018} and also the 
	connection to matroid theory developed in \cite{whitney_abstract_1935} has been discovered. 
\begin{definition}
	We say $S\subseteq \mathcal{L}$ is \emph{independent} in $\Gamma=(\mathcal{L},g,\mathcal{M})$
	if $S$ is linked into $\mathcal{M}$.
\end{definition}	
	
	The rank-nullity relation for the gammoid $\Gamma=(\mathcal{L},g,\mathcal{M})$ follows from the general 
	investigation of matroids in \cite{whitney_abstract_1935} and states, 
	that for any $S\in\mathcal{L}$
	\begin{equation}
		\text{rank}\, S + \text{null}\, S = \text{card}\, S \, ,
	\end{equation}
	where $\text{rank}\, S$ is the size of the largest independent subset $T\subseteq S$.
	In the main text, the spark of $\Gamma$ was defined as the largest integer such that
	\begin{equation}
		\text{card}\, S < \text{spark}\, \Gamma \Rightarrow \text{null}\, S = 0 \, .
	\end{equation}

	The following theorem is known for the spark of a matrix \cite{donoho_optimally_2003}. Utilizing gammoid 
	theory, we are able to proof its validity for non-linear dynamic systems.
	%
	%
	\begin{theorem}
		Consider a dynamic input-output system with input-output map $\Phi:\mathcal{U}\to\mathcal{Y}$ and 
		gammoid $\Gamma$.
		
		Let $\vec{y}\in\mathcal{Y}$ be given. If an input $\vec{u}\in\mathcal{U}$ solves
		\begin{equation}
			\Phi(\vec{u}) = \vec{y}
		\end{equation}
		and 
		\begin{equation}
		\Vert \vec{u} \Vert_0 < \frac{\text{spark}\, \Gamma}{2} \, ,
		\end{equation}
		 then for any other solution 
		$\vec{v}\in\mathcal{U}$ we find $\Vert \vec{v} \Vert_0 > \Vert \vec{u} \Vert_0$.
	\end{theorem}
	%
	%
	\begin{proof}
		We first reformulate the theorem as follows: There is at most one solution with ``$0$-norm'' 
		smaller than $\frac{\text{spark}\, \Gamma}{2}$. 
		
		So assume there are two distinct solutions 
		$\vec{u} \neq\vec{v}$ that both have ``$0$-norm'' smaller than $\frac{\text{spark}\, \Gamma}{2}$.
		If we denote $S:=\text{supp}\, \vec{u}$ and 
		$T:=\text{supp}\, \vec{v}$ the assumption says 
		\begin{equation}		
		\text{card}\, (S)<\frac{
		\text{spark}\, \Gamma}{2}
		\end{equation}		
		 and 
		 \begin{equation}
		 \text{card}\, (T)<\frac{ \text{spark}\, \Gamma}{2} \, .
		 \end{equation}		
		The union $Q:=S\cup T$ has $\text{card}\, Q < \text{spark}\, \Gamma$ thus by definition of the spark
		\begin{equation}
			\Phi: \mathcal{U}_Q \to \mathcal{Y}
		\end{equation}
		is invertible. So if there is a $\vec{w}\in\mathcal{U}_Q$ that solves
		\begin{equation}
			\Phi(\vec{w}) = \vec{y} \, ,
		\end{equation}		
		this $\vec{w}$ is unique with respect to $\mathcal{U}_Q$. 
		By construction $\mathcal{U}_S\subseteq \mathcal{U}_Q$, thus $\vec{u}$ is a solution that lies in 
		$\mathcal{U}_Q$. So we know that 
		$\vec{w}$ exists and is necessarily 
		equal to $\vec{u}$. But also $\mathcal{U}_T\subseteq \mathcal{U}_Q$ so $\vec{w}$ must also equal 
		$\vec{v}$. We found $\vec{u}=\vec{w}=\vec{v}$ which contradicts the assumption. 
	\end{proof}

\section{Convex Optimization}
	We provide the proof of theorem 9 from the main text.
	\begin{proof}
		We first show, that the constraint set
		\begin{equation}
			\mathcal{A} := \{ \vec{u}\in\mathcal{U} \, | \, \Vert \Phi(\vec{u})-\vec{y} \Vert_2 \leq
			 \epsilon \}
		\end{equation}
		with $\vec{y}\in \mathcal{Y}$ and $\epsilon > 0$ is convex. Let $\alpha,\beta > 0$ such that 
		$\alpha + \beta =1$. We have to show that 
		\begin{equation}
			\vec{u}, \vec{v} \in \mathcal{A} \, \Rightarrow \, (\alpha \vec{u} + 
			\beta \vec{v}) \in \mathcal{A} \, .
		\end{equation}
		Given that $\Phi$ is linear we can estimate
		\begin{equation}
			\Vert \Phi(\alpha \vec{u}+\beta \vec{u}) - \vec{y} \Vert_2 = 
			\Vert \alpha \Phi(\vec{u}) - \alpha \vec{y} + \beta \Phi(\vec{v}) -  \beta \vec{y} \Vert_2
		\end{equation}
		and with the triangle inequality
		\begin{equation}
			\Vert \Phi(\alpha \vec{u}+\beta \vec{u}) - \vec{y} \Vert_2  \leq 
			\alpha \Vert \Phi(\vec{u})- \vec{y} \Vert_2 + \beta \Vert \Phi(\vec{v})- \vec{y} \Vert_2 \leq
			\alpha \epsilon + \beta \epsilon = \epsilon \, .
		\end{equation}
		Since $\Vert \cdot \Vert_q$ is a proper norm on $\mathcal{U}$ we again apply the triangle 
		inequality
		\begin{equation}
			\Vert \alpha \vec{u} + \beta \vec{v} \Vert_1 \leq \alpha \Vert \vec{u} \Vert_1 + \beta \Vert \vec{v}
			\Vert_1 
		\end{equation}
		to see that the function we want to minimize is convex.
	\end{proof}
\section{Restricted Isometry Property}
	The Restricted Isometry Property first appeared in \cite{candes_decoding_2005} and was also 
	considered under 
	the name $\epsilon$-isometry in \cite{donoho_compressed_2006}. In the following we will consider 
	linear dynamic system. 
	%
	%
	\begin{definition}
		The Restricted Isometry Property of order $2k$ (RIP$2k$) is fulfilled if there is a constant 
		$\delta_{2k}$ 
		such that for $\vec{u},\vec{v}\in \Sigma_{k}$ the inequalities 
		\begin{equation}
			(1-\delta_{2k}) \Vert \underline{\vec{u}}-\underline{\vec{v}} \Vert_2^2 \leq 
			\Vert \underline{\Phi(\vec{u})} - \underline{\Phi(\vec{v})} \Vert_2^2 
		\end{equation}
		and
		\begin{equation}
			\Vert \underline{\Phi(\vec{u})} + \underline{\Phi(\vec{v})} \Vert_2^2 
			\leq 
			(1+\delta_{2k}) \Vert \underline{\vec{u}}+\underline{\vec{v}} \Vert_2^2 
		\end{equation}
		hold.
	\end{definition}
	The following lemma was proven for matrices in \cite{yonina_c_eldar_compressed_nodate} and can now 
	be generalized to linear dynamic systems.
	%
	%
	\begin{lemma}\label{lemma:delta2kinequality}
		Assume RIP$2k$ and let $\vec{u},\vec{v}$ with disjoint support. Then
		\begin{equation}
		\langle \underline{\Phi(\vec{u})  } , \underline{\Phi(\vec{v}) } \rangle \leq \delta_{2k} \Vert 
		\vec{u} \Vert_2 \Vert \vec{v} \Vert_2 \, . \label{eq:delta2kinequality}
		\end{equation}
	\end{lemma}
	%
	%
	\begin{proof}
		First we divide \eqref{eq:delta2kinequality} by $\Vert \vec{u}\Vert_2\Vert \vec{v}
		\Vert_2$ to get a normalized version of \eqref{eq:delta2kinequality}
		\begin{equation}
			\frac{ \langle \underline{\Phi(\vec{u})} , \underline{\Phi(\vec{v})} \rangle  }{ \Vert \vec{u}
			 \Vert_2 \Vert \vec{v}\Vert_2} =\left\langle \underline{
			 \Phi\left( \frac{\vec{u}}{ \Vert \vec{u}  \Vert_2} \right)}, \underline{
			 \Phi\left( \frac{\vec{v}}{ \Vert \vec{v}  \Vert_2} \right)		}	 
			 \right\rangle \leq \delta_{2k}
		\end{equation}
		where we used the homogeneity of $\underline{\,\cdot \,}$ and the linearity of $\Phi$. Note, that 
		underlined vectors are simply vectors in $\mathbb{R}^P$, hence the scalar product is the standard 
		scalar product. For simplicity 
		of notation we henceforth assume \\
		$\Vert \vec{u} \Vert_2 = \Vert \vec{v} \Vert_2 =1$. We apply the parallelogram identity to the 
		left hand 
		side of \eqref{eq:delta2kinequality} to get
		\begin{equation}
			\langle \underline{\Phi(\vec{u})} , \underline{\Phi(\vec{v})} \rangle = \frac{1}{4} \left(
			\Vert \underline{\Phi(\vec{u}) } + \underline{\Phi(\vec{v})} \Vert_2^2 - \Vert \underline{\Phi(
			\vec{u})} - \underline{\Phi(\vec{v})} \Vert_2^2
			\right)
		\end{equation}
		and since RIP$2k$ holds we get
		\begin{equation}
			\langle \underline{\Phi(\vec{u})},\underline{\Phi(\vec{v})} \rangle = \frac{1}{4} \left(
			(1+\delta_{2k}) \Vert \underline{\vec{u}} + \underline{\vec{v}} \Vert_2^2 - (1-\delta_{2k})
			\Vert \underline{\vec{u}} - \underline{\vec{v}} \Vert_2^2
			\right) \, .
		\end{equation}
		We apply lemma \ref{lemma:disjointSupport} to see that due to the disjoint support we get
		\begin{equation}
			\Vert \underline{\vec{u}} + \underline{\vec{v}} \Vert_2^2 =
			\Vert \underline{\vec{u}} - \underline{\vec{v}} \Vert_2^2 = 
			\Vert \vec{u} + \vec{v} \Vert_2^2
		\end{equation}
		and proposition \ref{prop:disjointSupport} yields
		\begin{equation}
			\langle \underline{\Phi(\vec{u})} ,\underline{ \Phi(\vec{v}) }\rangle \leq 2\delta_{2k} 
			\left(\Vert \vec{u} \Vert_2^2 +
			\Vert\vec{v}\Vert_2^2 \right) = \delta_{2k} \, .
		\end{equation}
	\end{proof}
	To see that our framework is in agreement with the results for the static problem consider the following:
	Let $A\in\mathbb{R}^{P \times N}$ and $\vec{y}\in\mathbb{R}^P$ be given and $\vec{w}\in\mathbb{R}^N$. 
	Solve
	\begin{equation}
		A\vec{w} = \vec{y}
	\end{equation}
	for $\vec{w}$. This problem can be seen as a dynamic system with trivial time-development and 
	input-output map $A$. 
	Consequently an input set $S=\{s_1,s_2,\ldots \}$ is independent in the gammoid 
	if 
	and only if the columns $\{ \vec{a}_{s_1},\vec{a}_{s_2},\ldots \}$ of $A$ are linearly independent. 
	In this 
	special case we can drop the underline notation and we can rewrite the RIP$2k$ condition as follows. Let 
	$\kappa:=2k$. For each $\vec{u},\vec{v}\in\Sigma_k$
	\begin{equation}
		(1-\delta_{2k}) \Vert \vec{u}-\vec{v} \Vert_2^2 \leq \Vert A\vec{u} - A\vec{v} \Vert_2^2
	\end{equation}
	is equivalent to saying that for each $\vec{w}:=\vec{u}-\vec{v}$ in $\Sigma_{2k}=\Sigma_\kappa$ 
	we have
	\begin{equation}
		(1-\delta_\kappa) \Vert \vec{w} \Vert_2^2 \leq \Vert A\vec{w} \Vert_2^2 \, .
	\end{equation}
	The same argumentation holds for the second inequality, so that we can say, the system has the RIP 
	of order $\kappa$ if and only if for all $\vec{w}\in\Sigma_\kappa$
	\begin{equation}
	(1-\delta_\kappa) \Vert \vec{w} \Vert_2^2 \leq \Vert A\vec{w} \Vert_2^2 
	\leq (1+\delta_\kappa) \Vert \vec{w} \Vert_2^2  \, .
	\end{equation}
	The latter is exactly the RIP condition that one usually formulates for static compressed sensing. We 
	can therefore say, that our framework contains the classical static problem as special case. 
	
	We now turn 
	our interest to one important proposition about systems that fulfil the RIP$2k$. This one corresponds to 
	Lemma 1.3 from \cite{yonina_c_eldar_compressed_nodate}. We follow the idea of the proof given there, 
	however, some additional steps are necessary in order to get a result valid for dynamic systems.
	%
	%
	\begin{proposition} \label{prop:alphabeta}
		Assume the RIP$2k$ holds and let $\vec{u}\in\mathcal{U}$ and $k\in \mathbb{N}$. Let 
		$\Lambda_0$ be an index set of size $\text{card}\, \Lambda_0 \leq k$. Let $\Lambda_1$ correspond 
		to the $k$ largest entries of $\underline{\vec{u}_{\Lambda_0^c}}$ (see lemma \ref{lemma:LambdaSums}), 
		$\Lambda_2$ to the second largest and so forth. We set $\Lambda:=\Lambda_0 \cup \Lambda_1$ and
		\begin{equation}
			\alpha := \frac{\sqrt{2}\delta_{2k}}{1-\delta_{2k}} \quad , \quad 
			\beta := \frac{1}{1-\delta_{2k}} \, .
		\end{equation}
		Then
		\begin{equation}
			\Vert \vec{u}_\Lambda \Vert_2 \leq \alpha \frac{\Vert \vec{u}_{\Lambda_0^c} \Vert_1}{\sqrt{k}}
			+\beta \frac{ \left\langle \underline{\Phi(\vec{u}_\Lambda)},
			\underline{\Phi(\vec{u})}\right\rangle }{
			\Vert \vec{u}_\Lambda \Vert_2}
			\, .
		\end{equation}
	\end{proposition}
	%
	%
	\begin{proof}

		For two complementary sets $\Lambda$ and $\Lambda^c$ we have 
		$\Lambda\cap \Lambda^c = \emptyset$ so that we can use equation
		\eqref{eq:LambdaDisjoint}. Furthermore, $\Lambda \cup \Lambda^c = \{1,\ldots,N\}$ shows that 
		$\vec{u}=\vec{u}_\Lambda+\vec{u}_{\Lambda^c}$. Thus due to linearity of $\Phi$ we get
		\begin{equation}
			\Phi(\vec{u}_\Lambda) = \Phi(\vec{u})-\Phi(\vec{u}_{\Lambda^c}) \, .
		\end{equation}
		By construction $\Lambda^c = \Lambda_2 \dot{\cup} \Lambda_3 \dot{\cup} \ldots$ so we can 
		substitute the latter term by a sum
		\begin{equation} \label{eq:T2}
			\Phi(\vec{u}_\Lambda) = \Phi(\vec{u})- \sum_{j\geq 2} \Phi(\vec{u}_{\Lambda_j}) \, .
		\end{equation}				
		By construction we also know that $\vec{u}_{\Lambda_0},\vec{u}_{\Lambda_1}\in\Sigma_k$, so from the 
		RIP$2k$ we get
		\begin{equation}
			(1-\delta_{2k}) \Vert \underline{\vec{u}_{\Lambda_0}} - \underline{\vec{u}_{\Lambda_1}} \Vert_2^2 
			\leq \Vert \underline{\Phi(\vec{u}_{\Lambda_0})}-\underline{\Phi(\vec{u}_{\Lambda_1})}   
			\Vert_2^2
		\end{equation}
		On the left hand side we notice that $\Lambda_0$ and $\Lambda_1$ are disjoint so we use 
		lemma \ref{lemma:disjointSupport} to see that
		\begin{equation}
			\Vert \underline{\vec{u}_{\Lambda_0}} - \underline{\vec{u}_{\Lambda_1}} \Vert_2^2 =
			\Vert \vec{u}_{\Lambda_0} + \vec{u}_{\Lambda_1} \Vert_2^2 = \Vert \vec{u}_\Lambda \Vert_2^2 \, .
		\end{equation}
		On the right hand side we use lemma \ref{lemma:2norminequality} and the linearity of $\Phi$ to get
		\begin{equation}
			\Vert \underline{\Phi(\vec{u}_{\Lambda_0})} - \underline{\Phi(\vec{u}_{\Lambda_1})}
			\Vert_2^2 \leq \Vert \Phi(\vec{u}_{\Lambda_0}) + \Phi (\vec{u}_{\Lambda_1}) \Vert_2^2 =
			\Vert \Phi(\vec{u}_{\Lambda}) \Vert_2^2 \, .
		\end{equation}
		We can estimate $\Vert \Phi(\vec{u}_\Lambda)\Vert_2^2$ by
		\begin{equation} \label{eq:T1}
			\Vert \Phi(\vec{u}_\Lambda)\Vert_2^2 \leq 
			\left\langle \underline{\Phi(\vec{u}_\Lambda)} , \underline{ \Phi(\vec{u})}
			\right\rangle 
			+\left\langle \underline{\Phi(\vec{u}_\Lambda)} ,  \sum_{j \geq 2} \underline{ 
			 \Phi(\vec{u}_{\Lambda_j}) }
			\right\rangle \, .
		\end{equation}
		To see that, we first wrote the norm as a scalar product
		\begin{equation}
			\Vert \Phi(\vec{u}_{\Lambda}) \Vert_2^2 = \left\langle \underline{\Phi(\vec{u}_\Lambda)} ,
			\underline{ \Phi(\vec{u})-\sum_{j\geq 2} \Phi(\vec{u}_{\Lambda_j})   } \right\rangle
		\end{equation}
		where the second vector comes from equation \eqref{eq:T2}.	
		The triangle inequality from lemma \ref{lemma:underlinezero} yields
		\begin{equation}
			\underline{ \Phi_i(\vec{u}) - \sum_{j \geq 2}
			 \Phi_i(\vec{u}_{\Lambda_j}) }
			 \leq 
			\underline{ \Phi_i(\vec{u})}	
			+\underline{ \sum_{j \geq 2} 
			 \Phi_i(\vec{u}_{\Lambda_j}) }
		\end{equation}				
		as well as 
		\begin{equation}
			\underline{ \sum_{j \geq 2} 
			 \Phi_i(\vec{u}_{\Lambda_j}) } \leq   \sum_{j \geq 2} \underline{
			 \Phi_i(\vec{u}_{\Lambda_j}) }
		\end{equation}					
		for each component $i=1,\ldots,P$. 	
		We proceed with the second scalar product in equation \eqref{eq:T1}
		\begin{equation}
			\left\langle \underline{ \Phi(\vec{u}_\Lambda)} ,
			\sum_{j\geq 2}  \underline{ \Phi(\vec{u}_{\Lambda_j})} \right\rangle =
			\left\langle \underline{\Phi(\vec{u}_{\Lambda_0}) } , 
			\sum_{j\geq 2} \underline{ \Phi(\vec{u}_{\Lambda_j}) } \right\rangle 
			+\left\langle \underline{ \Phi(\vec{u}_{\Lambda_1}) } , 
			\sum_{j\geq 2} \underline{ \Phi(\vec{u}_{\Lambda_j})} \right\rangle  
		\end{equation}
		where we again used the linearity and triangle inequality in each component,
		\begin{equation}
			\underline{\Phi_i(\vec{u}_\Lambda)}= \underline{ \Phi_i(\vec{u}_{\Lambda_0})
			+  \Phi_i(\vec{u}_{\Lambda_1})} \leq \underline{ \Phi_i(\vec{u}_{\Lambda_0}) } 
			+\underline{ \Phi_i(\vec{u}_{\Lambda_1})} \, .
		\end{equation}				
		For $m=0,1$ we first write the sum outside of the scalar product and apply lemma 
		\ref{lemma:delta2kinequality}
		\begin{equation}
			\left\langle \underline{\Phi(\vec{u}_{\Lambda_m})} ,
			\sum_{j\geq 2} \underline{ \Phi(\vec{u}_{\Lambda_j}) } \right\rangle =
			\sum_{j\geq 2}  \left\langle  \underline{ \Phi(\vec{u}_{\Lambda_m})} , \underline{
			\Phi(\vec{u}_{\Lambda_j})} \right\rangle 
			\leq \sum_{j\geq 2}
			\delta_{2k} \Vert \vec{u}_{\Lambda_m} \Vert_2\Vert \vec{u}_{\Lambda_j} \Vert_2
		\end{equation}
		and then lemma \ref{lemma:LambdaSums} to get
		\begin{equation}
			\left\langle \underline{ \Phi(\vec{u}_{\Lambda_m}) } , 
			\sum_{j\geq 2} \underline{ \Phi(\vec{u}_{\Lambda_j}) } \right\rangle \leq 
			\delta_{2k} \Vert \vec{u}_{\Lambda_m} \Vert_2 \frac{\Vert \vec{u}_{\Lambda_0^c}\Vert_1}{\sqrt{k}}
			\, .
		\end{equation}
		We add the two inequalities for $m=0$ and $m=1$ to get
		\begin{equation}
			\left\langle \Phi(\vec{u}_\Lambda) ,
			\sum_{j\geq 2} \Phi(\vec{u}_{\Lambda_j}) \right\rangle \leq 
			\delta_{2k} (\Vert \vec{u}_{\Lambda_0} \Vert_2 + \Vert \vec{u}_{\Lambda_1} \Vert_2 )
			\frac{\Vert \vec{u}_{\Lambda_0^c}\Vert_1}{\sqrt{k}} \, .
		\end{equation}
		From lemma \ref{lemma:sqrt2} we know that
		\begin{equation}
			\Vert \vec{u}_{\Lambda_0} \Vert_2 + \Vert \vec{u}_{\Lambda_1} \Vert_2 \leq 
			\sqrt{2}\, \Vert \vec{u}_{\Lambda_0} + \vec{u}_{\Lambda_1} \Vert_2 =\sqrt{2} \Vert 
			\vec{u}_{\Lambda} \Vert_2 \, .
		\end{equation}
		We combine all these results to get
		\begin{equation}
			(1-\delta_{2k}) \Vert \vec{u}_\Lambda \Vert_2^2 \leq \langle \underline{ \Phi(\vec{u}_\Lambda)},
			\underline{ \Phi(
			\vec{u} )} \rangle + \delta_{2k} \sqrt{2} \Vert \vec{u}_\Lambda \Vert_2 \frac{\Vert 
			\vec{u}_{\Lambda_0^c}
			\Vert_1}{\sqrt{k}} \, .
		\end{equation}
		For $\delta_{2k} < 1$ and with $\alpha$ and $\beta$ as defined above we can divide the inequality by 
		$(1-\delta)\Vert \vec{u}_\Lambda \Vert_2$ to get the desired inequality
		\begin{equation}
			\Vert \vec{u}_\Lambda \Vert_2 \leq \alpha \frac{\Vert \vec{u}_{\Lambda_0^c}\Vert_1}{\sqrt{k}} 
			+ \beta \frac{\langle \underline{\Phi(\vec{u}_\Lambda)} ,\underline{\Phi(\vec{u})} \rangle}{
			\Vert \vec{u}_\Lambda \Vert_2} \, .
		\end{equation}
	\end{proof}
	We want to make use of the latter proposition in the following way: 
	Say $\vec{v}\in\mathcal{U}$ is a fixed 
	vector, e.g., the true model error that we want to reconstruct, and $\vec{u}$ is our estimate for the 
	model 
	error, e.g., obtained by some optimization procedure. To measure how good the optimization procedure 
	performs, we compute the difference $\vec{w}:=\vec{u}-\vec{v}$ and utilize the proposition to get an 
	upper bound for $\Vert \vec{w} \Vert_2$
	
	Now assume the RIP$2k$ holds with a constant $\delta_{2k}< \sqrt{2}-1$. For the proposition we can choose 
	an arbitrary index set $\Lambda_0$ of size 
	$k$. We choose $\Lambda_0$ such that it corresponds to the $k$ components of $\underline{\vec{v}}$ with 
	highest magnitude. As 
	said in the proposition, $\Lambda_1$ will now correspond to the $k$ largest components in 
	$\underline{\vec{w}_{\Lambda_0^c}}$, $\Lambda_2$ to the second largest an so forth and we set 
	$\Lambda=\Lambda_0 \cup \Lambda_1$.
	
	Recap, that $\sigma_k(\vec{v})_q$ is the distance between $\vec{v}$ and the 
	the best $k$-sparse approximation \cite{foucart_mathematical_2013}
	\begin{equation}
		\sigma_k(\vec{v})_q := \min_{\tilde{\vec{v}}\in\Sigma_k} \Vert \tilde{\vec{v}}-\vec{v} \Vert_q 
	\end{equation}
	and note that
	\begin{equation}
		\sigma_k(\vec{v})_1 = \Vert \vec{v}_{\Lambda_0} - \vec{v} \Vert_1 \, . \label{eq:T5}
	\end{equation}
	To see the latter equation, let $\tilde{\vec{v}}\in\Sigma_k$ such that 
	$\Vert \tilde{\vec{v}} - \vec{v} \Vert_1 $ is minimal. 
	Since $\tilde{\vec{v}}$ is $k$-sparse, there is a $\tilde{\Lambda}$ such that $\tilde{\vec{v}}
	=\tilde{\vec{v}}_{\tilde{\Lambda}}$.
	Written as a sum 
	\begin{equation}
	 	\Vert \tilde{\vec{v}} - \vec{v} \Vert_1 = \sum_{i\in\tilde{\Lambda}} \underline{\tilde{v}_i - v_i}  +
	 	\sum_{i\in\tilde{\Lambda}^c} \underline{v_i} \, .
	\end{equation}
	The non-zero components must be chosen $\tilde{\vec{v}}_i=\vec{v}_i$ since this makes the 
	first sum vanish. In order to minimize the second sum the index set $\tilde{\Lambda}^c$ must 
	correspond to 
	the smallest $\underline{v_i}$, in other words, $\tilde{\Lambda}$ corresponds to the largest 
	$\underline{v_i}$, thus $\tilde{\vec{v}}=\vec{v}_{\Lambda_0}$. 
	
	In the remainder of this section we follow the argumentation line 
	of \cite{yonina_c_eldar_compressed_nodate} where the static problem in $\mathbb{R}^N$ was considered.
	Therein, many results are utilized that were first presented in \cite{candes_restricted_2008}.
	We show, how a line of reasoning can be made for the general case of composite Banach spaces, i.e., 
	valid for dynamic systems.
	%
	%
	\begin{lemma} \label{lemma:alphabetalemma}
		Consider the setting explained above and assume 
		$\Vert \vec{u}\Vert_1 \leq \Vert \vec{v}\Vert_1$. Then
		\begin{equation}
			\Vert \vec{w} \Vert_2 \leq 2 \Vert \vec{w}_\Lambda \Vert_2 + 2
			 \frac{\sigma_k(\vec{v})_1}{\sqrt{k}}
			\, .
		\end{equation}
	\end{lemma}
	%
	%
	\begin{proof}
		We begin with splitting $\vec{w}$ and applying the triangle inequality
		\begin{equation}
			\Vert \vec{w} \Vert_2 = \Vert \vec{w}_\Lambda + \vec{w}_{\Lambda^c} \Vert_2 \leq 
			\Vert \vec{w}_\Lambda \Vert_2 + \Vert \vec{w}_{\Lambda^c} \Vert_2 \, . 
		\end{equation}
		Since $\Lambda^c = \Lambda_2 \dot{\cup} \Lambda_3 \dot{\cup} \ldots$ we can apply the triangle inequality and 
		lemma \ref{lemma:LambdaSums}
		\begin{equation}\label{eq:T6}
			\Vert \vec{w}_{\Lambda^c} \Vert_2 = \left\Vert \sum_{j\geq 2} \vec{w}_{\Lambda_j} \right\Vert_2 
			\leq \sum_{j\geq 2}\Vert \vec{w}_{\Lambda_j} \Vert_2 \leq 
			 \frac{\Vert \vec{w}_{\Lambda_0^c} \Vert_1}{\sqrt{k}} \, .
		\end{equation}			
		By construction $\vec{u}=\vec{v}+\vec{w}$ so it is clear that
		\begin{equation}
			\Vert\vec{v} + \vec{w} \Vert_1 \leq \Vert \vec{v} \Vert_1 \, .
		\end{equation}
		For the complementary sets $\Lambda_0$ and $\Lambda_0^c$ we can apply 
		proposition \ref{prop:disjointSupport} with $q=1$ to get
		\begin{equation}\label{eq:T3}
			\Vert \vec{v} + \vec{w} \Vert_1 = \Vert \vec{v}_{\Lambda_0} + \vec{w}_{\Lambda_0} \Vert_1
			+  \Vert \vec{w}_{\Lambda_0^c} + \vec{v}_{\Lambda_0^c} \Vert_1 \, .
		\end{equation}
		In the proof of lemma \ref{lemma:2norminequality} we know for each component that
		\begin{equation}
			\vert \vert \underline{v_i}\vert - \vert \underline{w_i}\vert \vert \leq \vert 
			\underline{v_i+w_i} \vert \, .
		\end{equation}
		So we estimate the norm 
		\begin{equation} \label{eq:T4}
		\begin{aligned}
			\Vert & \vec{v}_{\Lambda_0} + \vec{w}_{\Lambda_0}  \Vert_1 = 
			\sum_{i\in \Lambda_0} \vert \underline{v_i+w_i} \vert \geq 
			\sum_{i\in\Lambda_0} \vert \vert \underline{v_i} \vert - \vert \underline{w_i} \vert \vert
			\\ & \geq \left\vert \sum_{i\in\Lambda_0} (\vert \underline{v_i} \vert - \vert \underline{w_i} \vert )  	
			\right\vert 
			= \vert (\Vert \vec{v}_{\Lambda_0} \Vert_1 - \Vert \vec{w}_{\Lambda_0} \Vert_1) \vert
			\geq \Vert \vec{v}_{\Lambda_0} \Vert_1 - \Vert \vec{w}_{\Lambda_0} \Vert_1 
			\end{aligned}
		\end{equation}
		and the same holds for $\Lambda_0^c$. With this result equation \eqref{eq:T3} becomes
		\begin{equation}
			\Vert \vec{v}_{\Lambda_0} \Vert_1 - \Vert \vec{w}_{\Lambda_0} \Vert_1 
			+\Vert \vec{w}_{\Lambda_0^c} \Vert_1 - \Vert \vec{v}_{\Lambda_0^c} \Vert_1 \leq \Vert \vec{v} \Vert_1
		\end{equation}
		which yields a lower bound for $\Vert \vec{w}_{\Lambda_0^c} \Vert_1$
		\begin{equation}
			\Vert \vec{w}_{\Lambda_0^c} \Vert_1  \leq 
			\Vert \vec{v} \Vert_1 -\Vert \vec{v}_{\Lambda_0} \Vert_1 + \Vert 
			\vec{w}_{\Lambda_0} \Vert_1 + \Vert \vec{v}_{\Lambda_0^c} \Vert_1 \, .
		\end{equation}		 
		A calculation as in equation \eqref{eq:T4} and equation \eqref{eq:T5} show that
		\begin{equation}
			\Vert \vec{v} \Vert_1 -\Vert \vec{v}_{\Lambda_0} \Vert_1 \leq 
			\Vert \vec{v} - \vec{v}_{\Lambda_0} \Vert_1 = \sigma_k(\vec{v})_1 
		\end{equation}
		and since $\Lambda_0^c$ is complementary to $\Lambda_0$, 
		$\vec{v}=\vec{v}_{\Lambda_0}+\vec{v}_{\Lambda_0^c}$, thus
		\begin{equation}
			\Vert \vec{v}_{\Lambda_0^c} \Vert_1 = \Vert \vec{v} - \vec{v}_{\Lambda_0} \Vert_1 = \sigma_k
			(\vec{v})_1 \, .
		\end{equation}
		Combining the latter three results we get
		\begin{equation}
			\Vert \vec{w}_{\Lambda_0^c} \Vert_1 \leq \Vert \vec{w}_{\Lambda_0} \Vert_1 + 2 \sigma_k(\vec{v})_1 
			\, .
		\end{equation}
		Inserting the latter result into equation \eqref{eq:T6} yields
		\begin{equation}
			\Vert \vec{w}_{\Lambda^c} \Vert_2 \leq \frac{ \Vert \vec{w}_{\Lambda_0} \Vert_1 +
			2 \sigma_k(\vec{v})_1}{\sqrt{k}} \, .
		\end{equation}
		From lemma \ref{lemma:sqrtk} we get
		\begin{equation} \label{eq:T7}
			\Vert \vec{w}_{\Lambda^c} \Vert_2 \leq \Vert \vec{w}_{\Lambda_0} \Vert_2 + 2 
			\frac{ \sigma_k(\vec{v})_1}{\sqrt{k}} \, .
		\end{equation}
		We now use the triangle inequality
		\begin{equation}
			\Vert \vec{w} \Vert_2 \leq \Vert \vec{w}_\Lambda \Vert_2 + \Vert \vec{w}_{\Lambda^c} \Vert_2
		\end{equation}
		and the fact, that $\Lambda_0 \subseteq \Lambda$, thus
		\begin{equation}
			\Vert \vec{w}_{\Lambda_0} \Vert_2 \leq \Vert \vec{w}_{\Lambda} \Vert_2
		\end{equation}
		to get the desired inequality
		\begin{equation}
			\Vert \vec{w} \Vert \leq 2 \Vert \vec{w}_\Lambda \Vert_2 + 
			 2 \frac{ \sigma_k(\vec{v})_1}{\sqrt{k}} \, .
		\end{equation}
	\end{proof}
	The following theorem follows from a combination of proposition \ref{prop:alphabeta} and
	lemma \ref{lemma:alphabetalemma}. It is the key result for the optimization 
	problem we present for linear dynamic input-output systems.
	%
	%
	\begin{theorem} \label{theorem:C0C1}
		Let $\vec{u},\vec{v}\in\mathcal{U}$ with $\Vert \vec{u} \Vert_1 \leq \Vert \vec{v}\Vert_1$ and 
		assume for $\Phi$ we have the RIP$2k$ with $\delta_{2k}< \sqrt{2}-1$. Let $\Lambda_0$ correspond
		to the 
		$k$ largest components of $\underline{v}$. We set $\vec{w}:=\vec{u}-\vec{v}$ and let $\Lambda_1$ 
		correspond to the $k$ largest components of $\underline{\vec{w}_{\Lambda_0^c}}$, $\Lambda_2$ to the 
		second largest and so forth. Let $\Lambda:= \Lambda_0 \cup \Lambda_1$. There are two constants 
		$C_0$ and 
		$C_1$ such that
		\begin{equation}
			\Vert \vec{w} \Vert_2 \leq C_0 \frac{\sigma_k(\vec{v})_1}{\sqrt{k}} + C_1 
			\frac{ \langle \underline{\Phi(\vec{w}_\Lambda)} , \underline{\Phi(\vec{w})} \rangle  }{\Vert 
			\vec{w}_\Lambda \Vert_2} \, .
		\end{equation}
	\end{theorem}
	%
	%
	\begin{proof}
		We start with proposition \ref{prop:alphabeta} applied to $\vec{w}$
		\begin{equation}
			\Vert \vec{w}_\Lambda \Vert_2 \leq \alpha \frac{\Vert \vec{w}_{\Lambda_0^c} 
			\Vert_1}{\sqrt{k} } + \beta 
			\frac{ \langle \underline{\Phi(\vec{w}_\Lambda)} , \underline{\Phi(\vec{w})} \rangle  }{\Vert 
			\vec{w}_\Lambda \Vert_2} \, .
		\end{equation}
		In lemma \ref{lemma:alphabetalemma} equation \ref{eq:T7} we found an estimate of 
		$\Vert \vec{w}_{\Lambda_0^c} \Vert_1$, inserted into the latter inequality
		\begin{equation}
			\Vert \vec{w}_\Lambda \Vert_2 \leq  \alpha \frac{\Vert \vec{w}_{\Lambda_0} \Vert_1}{\sqrt{k}} + 
			2\alpha \frac{\sigma_k(\vec{v})_1}{\sqrt{k}} + \beta 
			\frac{ \langle \underline{\Phi(\vec{w}_\Lambda)} , \underline{\Phi(\vec{w})} \rangle  }{\Vert 
			\vec{w}_\Lambda \Vert_2} \, .
		\end{equation}
		The first term can be treated with lemma \ref{lemma:sqrtk}
		and then we use the fact that $\Lambda_0\subseteq \Lambda$, thus 
		$\Vert \vec{w}_{\Lambda_0}\Vert_2 \leq 
		\Vert \vec{w}_\Lambda \Vert_2$, to get
		\begin{equation}
			\Vert \vec{w}_\Lambda \Vert_2 \leq  \alpha \Vert \vec{w}_{\Lambda} \Vert_2 + 
			2\alpha \frac{\sigma_k(\vec{v})_1}{\sqrt{k}} + \beta 
			\frac{ \langle \underline{\Phi(\vec{w}_\Lambda)} , \underline{\Phi(\vec{w})} \rangle  }{\Vert 
			\vec{w}_\Lambda \Vert_2} \, .
		\end{equation}
		Due to the assumption $\delta_{2k}< \sqrt{2}-1$ we also get $\alpha < \sqrt{2}-1 <1$ hence 
		$(1-\alpha)$ is 
		positive so we rewrite the latter equation as
		\begin{equation}
			\Vert \vec{w}_\Lambda \Vert_2 \leq  
			\frac{2\alpha}{1-\alpha} \frac{\sigma_k(\vec{v})_1}{\sqrt{k}} + \frac{\beta }{1-\alpha}
			\frac{ \langle \underline{\Phi(\vec{w}_\Lambda)} , \underline{\Phi(\vec{w})} \rangle  }{\Vert 
			\vec{w}_\Lambda \Vert_2} \, .
		\end{equation}
		To close the proof we use lemma  \ref{lemma:alphabetalemma} to get
		\begin{equation}
			\frac{1}{2} \left( \Vert \vec{w} \Vert - 2 \frac{\sigma_k(\vec{v})_1}{\sqrt{k}} \right)
			 \leq  
			\frac{2\alpha}{1-\alpha} \frac{\sigma_k(\vec{v})_1}{\sqrt{k}} + \frac{\beta }{1-\alpha}
			\frac{ \langle \underline{\Phi(\vec{w}_\Lambda)} , \underline{\Phi(\vec{w})} \rangle  }{\Vert 
			\vec{w}_\Lambda \Vert_2} 
		\end{equation}
		and with 
		\begin{equation}
			C_0 := \left(
				\frac{4\alpha}{1-\alpha} +2
			\right)
		\end{equation}
		and 
		\begin{equation}
			C_1:= \frac{2\beta}{1-\alpha}
		\end{equation}
		we finally obtain
		\begin{equation}
			\Vert \vec{w} \Vert_2 \leq C_0 \frac{\sigma_k(\vec{v})_1}{\sqrt{k}} + C_1 
			\frac{ \langle \underline{\Phi(\vec{w}_\Lambda)} , \underline{\Phi(\vec{w})} \rangle  }{\Vert 
			\vec{w}_\Lambda \Vert_2} \, .
		\end{equation}
	\end{proof}
	We can now apply theorem \ref{theorem:C0C1} to the solutions of the $\Vert \cdot \Vert_0$ and 
	$\Vert \cdot 
	\Vert_1$ optimization problems. Assume $\vec{y}^\text{data}\in\mathcal{Y}$ is given data which is 
	produced by a sparse ``true'' input 
	$\vec{w}^*\in \mathcal{U}$, i.e.,
	\begin{equation}
		\Phi(\vec{w}^*) = \vec{y}^\text{data} \, .
	\end{equation}
	We want to infer $\vec{w}^*$ from $\vec{y}^\text{data}$.
	Sparsity of the ``true'' input $\vec{w}^*$ means, that for all $\vec{u}$ in
	\begin{equation}
		\mathcal{A} := \left\{ \vec{u} \in \mathcal{U} \, | \, \Vert \Phi(\vec{u}) - \vec{y} \Vert_2 = 0
		\right\}
	\end{equation}
	we find 
	\begin{equation}
		\Vert \vec{w}^* \Vert_0 \leq \Vert \vec{u} \Vert_0 \, .
	\end{equation}
	Now let $\hat{\vec{w}}$ be a solution of the convex $\Vert \cdot \Vert_1$ optimization problem, that is, 
	for all 
	$\vec{u}\in\mathcal{A}$ we find
	\begin{equation}
		\Vert \hat{\vec{w}} \Vert_1 \leq \Vert \vec{u} \Vert_1 \, .
	\end{equation}
	We can now apply theorem \ref{theorem:C0C1} to $\vec{w}=\hat{\vec{w}}-\vec{w}^*$. Note, that 
	\begin{equation}
		\Phi(\hat{\vec{w}}) = \Phi(\vec{w}^*)\, \Rightarrow \,\Phi(\vec{w})=0 \, .
	\end{equation}		
	We find
	\begin{equation}
		\Vert \hat{\vec{w}}-\vec{w}^* \Vert_2 \leq C_0 \frac{\sigma_k(\vec{w}^*)_1}{\sqrt{k}} \, .
	\end{equation}
	Note that by definition $\sigma_k(\vec{w}^*)_1$ is the best $k$-sparse approximation to $\vec{w}^*$ in 
	$1$-norm. Thus we have shown, that the convex $\Vert \cdot\Vert_1$ optimization yields an 
	approximation of 
	the unknown ``true'' input.
	
	Theorem \ref{theorem:C0C1} might be adjusted for various scenarios where we can make further assumptions 
	about the model error $\vec{w}^*$ or about stochastic or measurement errors. We want to derive a last 
	inequality for the case of bounded noise. Let $\xi$ represent the noise, and $\Phi_\text{noiseless}$ the 
	solution operator we have discussed so far. We now consider a new solution operator
	\begin{equation}
		\Phi(\vec{u})(t) := \Phi_\text{noiseless}(\vec{u})(t) + \xi(t) 
	\end{equation}
	that incorporates the noise $\xi$. Since $\xi(t)\in\mathbb{R}^P$ we can interpret 
	$\xi\in \mathcal{Y}$ and use the norm on $\mathcal{Y}$. We assume that $\xi$ is a bounded noise with 
	$\epsilon > 0$,
	\begin{equation}
		\Vert \xi \Vert_2 \leq \epsilon \, .
	\end{equation}
	We adjust the solution set
	\begin{equation}
		\mathcal{A} := \left\{ \vec{u} \in \mathcal{U} \, | \, \Vert \Phi(\vec{u}) - \vec{y} \Vert_2 
		\leq \epsilon
		\right\} 
	\end{equation}
	and define $\vec{w}^*,\hat{\vec{w}} \in \mathcal{A}$ as before with minimal $\Vert \cdot \Vert_0$ and 
	$\Vert \cdot \Vert_1$ norm, respectively.	
	From the theorem we get
	\begin{equation} \label{eq:T9}
			\Vert \vec{w} \Vert_2 \leq C_0 \frac{\sigma_k(\vec{w}^*)_1}{\sqrt{k}} + C_1 
			\frac{ \langle \underline{\Phi(\vec{w}_\Lambda)} , \underline{\Phi(\vec{w})} \rangle  }{\Vert 
			\vec{w}_\Lambda \Vert_2} \, .
	\end{equation}
	We want to estimate the scalar product by the Cauchy-Schwartz inequality
	\begin{equation} \label{eq:T8}
		\langle \underline{\Phi(\vec{w}_\Lambda)} , \underline{\Phi(\vec{w})} \rangle 
		\leq \Vert \Phi(\vec{w}_\Lambda) \Vert_2 \Vert \Phi(\vec{w}) \Vert_2 \, .
	\end{equation}
	By construction $\Lambda=\Lambda_0 \cup \Lambda_1$ has at most $2k$ elements. thus, it is possible to write 
	$\vec{w}=\vec{u}+\vec{v}$ where $\vec{u},\vec{v}\in \Sigma_k$ have disjoint support. With lemma
	\ref{lemma:underlinezero} we get
	\begin{equation}
		\Vert \Phi(\vec{w}_\Lambda)\Vert_2 \
		=\Vert \underline{\Phi(\vec{u}) + \Phi(\vec{v})} 
		\Vert_2		
		\leq \Vert \underline{\Phi(\vec{u})} + \underline{\Phi(\vec{v})} 
		\Vert_2 \, .
	\end{equation}
	We can now use the RIP$2k$ to get
	\begin{equation}
		 \Vert \Phi(\vec{w}_\Lambda)\Vert_2 \leq 
		 \sqrt{1+ \delta_{2k}} \Vert \underline{\vec{u}} + \underline{\vec{v}} \Vert_2
	\end{equation}
	and with lemma \ref{lemma:disjointSupport}
	\begin{equation}
		 \Vert \Phi(\vec{w}_\Lambda)\Vert_2 \leq 
		 \sqrt{1+ \delta_{2k}} \Vert \vec{u} + \vec{v} \Vert_2 =\sqrt{1+ \delta_{2k}} \Vert 
		 \vec{w}_\Lambda \Vert_2 \, .
	\end{equation}
	On the other hand we can write $\vec{w}= \hat{\vec{w}}-\vec{w}^*$ and get
	\begin{equation}
		\Vert \Phi(\vec{w}) \Vert = \left\Vert \left( \Phi(\hat{\vec{w}}) - \vec{y} \right) - \left( 
		\Phi(\vec{w}^*) - \vec{y} \right) \right\Vert_2 \leq
		 \left\Vert  \Phi(\hat{\vec{w}}) - \vec{y}  \right\Vert_2 + 
		 \left\Vert  
		\Phi(\vec{w}^*) - \vec{y} \right\Vert_2 \leq 2 \epsilon \, .
	\end{equation}
	With this, equation \eqref{eq:T8} becomes
	\begin{equation}
		\langle \underline{\Phi(\vec{w}_\Lambda)} , \underline{\Phi(\vec{w})} \rangle  \leq
		2 \epsilon \sqrt{1+\delta_{2k}} \Vert \vec{w}_\Lambda \Vert_2 
	\end{equation}
	and inserting this into equation \eqref{eq:T9} leads to
	\begin{equation}
		\Vert \vec{w} \Vert_2 \leq C_0 \frac{\sigma_k(\vec{w}^*)_1}{\sqrt{k}} + C_1 
			2 \epsilon \sqrt{1+\delta_{2k}} \, .
	\end{equation}
	We can now adjust the constant $C_2:= \sqrt{1+\delta_{2k}} C_1 $ to get the result
	\begin{equation}
	\Vert \hat{\vec{w}}-\vec{w}^* \Vert_2 \leq C_0 \frac{\sigma_k(\vec{w}^*)_1}{\sqrt{k}} + C_2
		 \epsilon 
	\end{equation}
	which is the equation from theorem 4 of the main text. 
	
\section{A Note on Weighted Gammoids}
	The edge weights of an influence graph are defined by
	\begin{equation}
		F(i\to j) := \frac{\partial  f_j(\vec{x})}{\partial x_i} \, .
	\end{equation}
	Since we only consider linear systems in this section, the edge weights are real valued constants. It is easy to see, 
	that there is a one-to-one correspondence between linear input-output systems and 
	weighted gammoids. However, 
	we want to emphasize, that the construction of the weighted gammoid is not restricted to linear systems, as 
	soon as we allow for state dependent weights, see for instance \cite{wey_rank_1998}. The investigation 
	of non-linear gammoids might yield 
	an approach 
	to a non-linear extension of our framework and shall be considered in future research.

	From \cite{wey_rank_1998} we deduced, that the gammoid of a system contains the information about the 
	structurally invertible input-output configurations. But we also now, that 
	structural properties are not sensitive to numerical ill-posedness. This is the reason why we have 
	examined the possibility of a convex optimization as well as the RIP$2k$ property. Those, however, are 
	hard to verify.	
	Fortunately we can again make use of the gammoid interpretation.
\subsection{Transfer Function} \label{sec:Transfer}
	The following lemma is known in the literature (see e.g. \cite{newman_networks_2010}), but usually given 
	for adjacency matrices with entries that are either zero or one. 
	We formulate it in a way that is consistent with our notation and such that the edge weights can be 
	arbitrary.
	%
	%
	\begin{lemma} \label{lemma:Ak}
		Let $A\in\mathbb{R}^{N\times N}$ a matrix and $g=(\mathcal{N},\mathcal{E})$ the weighted graph 
		with nodes $\mathcal{N}=\{1,\ldots,N\}$ and edges $(i\to j)\in\mathcal{E}$ whenever $A_{ji}\neq 0$. For 
		each edge we define its weight as $F(i\to j):= A_{ji}$. The edge weights imply a weight 
		for sets of paths. Let $\mathcal{P}_k(a,b)$ denote the set of 
		paths from node $a$ to node $b$ of length $k$. Then
		\begin{equation}
			A^k_{ba} = F(\mathcal{P}_k(a,b))
		\end{equation}
	\end{lemma}	
	%
	%
	\begin{proof}
		Let $l_0,l_1,\ldots,l_k\in\mathcal{N}$ be a list of nodes. If the path 
		$\pi:=(l_0\to \ldots \to l_k)$ 
		exists, then $\pi \in \mathcal{P}_k(l_0,l_k)$ and we can use the homomorphic property of $F$ to get
		\begin{equation}
			F(\pi) = F(l_0\to l_1)\ldots F(l_{k-1}\to l_k) = A_{l_kl_{k-1}} \ldots A_{l_1 l_0} \, .
		\end{equation}
		On the other hand, if $\pi$ does not exist, that means at least one of the terms $A_{l_il_{i-1}}$ 
		equals zero and
		\begin{equation}
			 A_{l_kl_{k-1}} \ldots A_{l_1 l_0} = 0 \, .
		\end{equation}
		When we compute the powers of $A$ we find
		\begin{equation}
			A^k_{ba} = \sum_{l_1,\ldots, l_{k-1}=1}^N A_{b l_{k-1}} \ldots A_{l_1 a}
		\end{equation}
		which sums up all node lists $l_1,\ldots,l_{k-1}\in\mathcal{N}$ with $l_0=a$ and $l_k=b$ fixed. It 
		is clear, that the terms in the sum do not vanish if and only if the path 
		$(a\to l_1 \to \ldots \to l_{k-1} \to b)$ exists. Thus we can replace the sum by
		\begin{equation}
			A^k_{ba} = \sum_{\pi \in \mathcal{P}_k(a,b)}^N A_{b l_{k-1}} \ldots A_{l_1 a}
		\end{equation}
		and we have already seen, that for an existing path we can replace the right hand side by 
		\begin{equation}
			A^k_{ba} = \sum_{\pi \in \mathcal{P}_k(a,b)} F(\pi) = F(\mathcal{P}_k(a,b))
		\end{equation}
		where the second equality comes simply from the definition of the weight function for sets of paths.
	\end{proof}
	In Laplace space, a linear dynamic input-output system takes the form
	\begin{equation}
		\tilde{\vec{y}}(\sigma) = T(\sigma) \tilde{\vec{w}}(\sigma)
	\end{equation}
	where $\sigma\in\mathbb{C}$ is a complex variable, $\tilde{\vec{w}}:\mathcal{T}$ and 
	$\tilde{\vec{y}}$ are the 
	Laplace transforms of the input $\vec{w}$ and output $\vec{y}$, 
	respectively, see for instance \cite{luenberger_introduction_1979}. One will realize, that 
	the transfer function $T$ is the Laplace representation of the operator $\Phi$.
	
	We have earlier discussed, that for linear systems, an input set $S$ can either be understood as a 
	restriction to the input space $\mathcal{U}$, or, equivalently, as a submatrix of the solution operator 
	$\Phi$ (in state space) and clearly also $T$ (in Laplace space).
	The differential equation of a linear system
	\begin{equation}
		\dot{\vec{x}}(t) = A \vec{x}(t) + \vec{w}(t)
	\end{equation}	
	can be written after Laplace transformation as
	\begin{equation}
		\tilde{\vec{x}}(\sigma) = (\sigma-A)^{-1} \tilde{\vec{w}}(\sigma) \, .
	\end{equation}
	We can interpret 
	\begin{equation}
		T^\text{full}(\sigma)=(\sigma-A)^{-1}
	\end{equation}		
	as the \emph{full transfer function}. For a given input set $S$ and output set $Z$ we can then simply take 
	the columns indicated by $S$ an rows indicated by $Z$ to get the transfer function for this specific 
	configuration.
	%
	%
	\begin{proposition}\label{prop:Transferfunction}
		 Let $T^\text{full}$ be the full transfer function of a linear system and $i,j\in\mathcal{N}$ nodes 
		 in the weighted influence graph. Let $\mathcal{P}(i,j)$ denote the paths from $i$ to $j$. Then
		\begin{equation}
			T^\text{full}_{ji}(\sigma) = \frac{1}{\sigma} \sum_{\pi \in \mathcal{P}(i,j)} 
			\frac{F(\pi)}{\sigma^{\text{len}\,\pi}}
			\, .
		\end{equation}
	\end{proposition}	
	%
	%
	\begin{proof}
		We use the Neumann-series to write
		\begin{equation}
			T_{ji}^\text{full}(\sigma) = \left[ (\sigma-A)^{-1} \right]_{ji} =\frac{1}{\sigma}
			 \left[\left(1-\frac{A}{\sigma}
			\right)^{-1}\right]_{ji} 
			= \frac{1}{\sigma}
			\sum_{k=0}^\infty \frac{ \left[A^k\right]_{ji}}{\sigma^k} \, .
		\end{equation}
		where the brackets just indicate that we first take the matrix power and then take the $ji$-th element. 
		With lemma \ref{lemma:Ak} we get
		\begin{equation}
			T^\text{full}_{ji}(\sigma) = \frac{1}{\sigma} \sum_{k=0}^\infty \sum_{\pi\in\mathcal{P}_k(i,j)} 
			\frac{F(\pi)}{\sigma^k} \, .
		\end{equation}
		We now see that in $\sigma^{-k}$ we always find that $k$ is the length of the path, $k=\text{len}\, \pi$. 
		Furthermore we can combine the two sums to one sum over all paths
		\begin{equation}
			T^\text{full}_{ji}(\sigma) = \frac{1}{\sigma}  \sum_{\pi\in\mathcal{P}(i,j)} 
			\frac{F(\pi)}{\sigma^k} \, .
		\end{equation}
	\end{proof}
\subsection{Transposed Gammoids}
	For a linear dynamic system
	\begin{equation}
	\begin{aligned}
		\dot{\vec{x}}(t) &= A\vec{x}(t) + B \vec{w}(t) \\
		\vec{x}(0) &= \vec{x}_0 \\
		\vec{y}(t) &= C\vec{x}(t)	
	\end{aligned}
	\end{equation}
	the system
	\begin{equation}
	\begin{aligned}
		\dot{\vec{x}}(t) &= A^T\vec{x}(t) + C^T \vec{y}(t) \\
		\vec{x}(0) &= \vec{x}_0 \\
		\vec{w}(t) &= B^T \vec{x}(t)	 
	\end{aligned}
	\end{equation}
	is called \textit{dual} in the literature, referring to the duality principle of optimization theory, see 
	for instance \cite{lunze_einfuhrung_2016}. 
	To avoid confusion, we will use the term \emph{transposed system} instead, which then leads to a \emph{transposed gammoid}. 
	This nomenclature helps avoiding confusion with the term \textit{dual gammoid}, which is already occupied by the
	duality principle of matroid theory \cite{whitney_abstract_1935} and which, to the best of our knowledge,
	is not related to dual dynamic systems.
	
	From a gammoid $\Gamma$ one can easily switch to the transposed gammoid $\Gamma'$ without detour over the 
	transposed dynamic system. Let $g$ and $g'$ denote the influence graphs of the original and the transposed 
	system, respectively. Both systems have differential equations for the variables $\vec{x}=(x_1,\ldots,
	x_N)^T$, so both influence graphs have $N$ nodes. To avoid confusion, say 
	\begin{equation}
		\mathcal{N}=\{1,\ldots,N\}
	\end{equation}		
	and the nodes of the transposed graph are indicated by a prime symbol 
	\begin{equation}
		\mathcal{N}' = \{1',\ldots,N' \} \, . 
	\end{equation}
	We also have a one-to-one correspondence between the edges $\mathcal{E}$ and $\mathcal{E}'$ which can be 
	written as
	\begin{equation}
		(i\to j)' = (j' \to i')
	\end{equation}		
	which is the gammoid analogue to $A_{ji}=(A^T)_{ij}$. This shows, that to switch from $g$ to $g'$, one 
	basically has to flip the edges. Finally, one will realize, that inputs and outputs swap their roles. So 
	if $S=(s_1,\ldots,s_M)$ is an input set in $g$, then $S'=(s_1',\ldots,s_M')$ is an output set of $g'$. 
	In the same way, an output set $Z$ in $g$ becomes an input set $Z'$ in $g'$. We find that
	\begin{equation}
		(\mathcal{L},g,\mathcal{M})' = (\mathcal{M}',g',\mathcal{L}')
	\end{equation}
	directly maps the gammoid $\Gamma$ to the transposed gammoid $\Gamma'$. We have just derived the construction 
	of the transposed gammoid for linear system. Clearly, we can use the derived formula to define the 
	transposed gammoid also for the general case.
	%
	%
	\begin{definition}
		Let $\Gamma=(\mathcal{L},g,\mathcal{M})$ be a gammoid, we call 
		\begin{equation}
			\Gamma' := (\mathcal{M}',g',\mathcal{L}')
		\end{equation}
		the transposed gammoid.
	\end{definition}
	If the gammoid is a weighted gammoid, one simply {keeps} the weight of each single edge, 
	\begin{equation}
		F( i\to j ) = F( (i \to j)' ) \, .
	\end{equation}
	Now any path 
	\begin{equation}
		\pi = (n_0\to \ldots \to n_k)
	\end{equation}
	in $\Gamma$ corresponds to a path
	\begin{equation}
		\pi' = (n_k' to \ldots  n_0')
	\end{equation}
	in $\Gamma'$ and
	\begin{equation}
		F(\pi) = F(\pi') \, .
	\end{equation}

\subsection{Concatenation of Gammoids}
	Consider two gammoids $\Gamma_1=(\mathcal{L}_1,g_1,\mathcal{M}_1)$ and 
	${\Gamma}_2=({\mathcal{L}}_2,{g_2},{\mathcal{M}_2})$. If we now think of a signal that flows through 
	a gammoid from the inputs to the outputs we can define the 
	\emph{concatenation}
	\begin{equation}
		\Gamma := \Gamma_1 \circ \Gamma_2 \, .
	\end{equation}		
	A signal enters somewhere in the input ground set 
	$\mathcal{L}_1$ and flows through $\Gamma_1$ to the output ground set $\mathcal{M}_1$. From there, the 
	signal is passed over to $\mathcal{L}_2$ and flows through $\Gamma_2$ to the final output ground set 
	$\mathcal{M}_2$. To get a well defined concatenation, one must assume, that the signal transfer between the 
	gammoids, i.e., from $\mathcal{M}_1$ to $\mathcal{L}_2$ is well defined. To ensure this, we assume, that 
	$\mathcal{M}_1=\{m_1,\ldots, m_P\}$ and $\mathcal{L}_2=\{l_1,\ldots, l_P\}$ have the same size and are 
	present in a fixed ordering such that the signal is passed from $m_i$ to $l_i$. In other words, we identify 
	\begin{equation}
		m_i = l_i \quad \text{for} \quad i=1,\ldots ,P \, .
	\end{equation}
	The result $\Gamma=(\mathcal{L},g,\mathcal{M})$ is indeed again a gammoid with input ground set 
	$\mathcal{L}=\mathcal{L}_1$ and output ground set $\mathcal{M}=\mathcal{M}_2$. The graph $g$ of this 
	gammoid is simply the union of $g_1$ and $g_2$, i.e., 
	$g=(\mathcal{N}_1\cup \mathcal{N}_2, \mathcal{E}_1\cup \mathcal{E}_2)$.
	%
	%
	\begin{proposition}
		Let $\Gamma_1=(\mathcal{L}_1,g_1,\mathcal{M}_1)$ and ${\Gamma}_2=({\mathcal{L}}_2,{g_2},{\mathcal{M}_2})$ 
		be two gammoids, $\Gamma=(\mathcal{L},g,\mathcal{M})$ the result of the concatenation $\Gamma_1\circ 
		\Gamma_2$.
		For any $S\subseteq \mathcal{L}$ and $K\subseteq \mathcal{M}$ we find: $S$ is linked in $g$ to $K$ if 
		and only if there is a 
		$Z\subseteq \mathcal{M}_1=\mathcal{L}_2$ such that $S$ is linked in $g_1$ to $Z$ and $Z$ is linked in 
		$g_2$ to $K$.
	\end{proposition}
	%
	%
	\begin{proof}
		First assume, that $S\subseteq \mathcal{L}$ is linked in $g$ to $K\subseteq \mathcal{M}$ that is, there 
		is a family of node-disjoint paths $\Pi = \{\pi_1,\ldots , \pi_M \}$ from $S$ to $\mathcal{M}$, where 
		$M=\text{card}\, S=\text{card}\,K$.	By construction, for any edge $i\to j$ with $i\in\mathcal{N}_1$ 
		and $j\in\mathcal{N}_2$ we necessarily find that either $i \in \mathcal{M}_1$ or $j\in\mathcal{M}_1$. 
		Since $S\subseteq\mathcal{L}_1\subseteq \mathcal{N}_1$ and $K\subseteq \mathcal{M}_2\subseteq 
		\mathcal{N}_2$ it is clear, that each path $\pi_i$ has at least one node $z_i\in\mathcal{M}_1$ and these 
		nodes are pairwise distinct. We find that $Z:=\{z_1,\ldots, z_M\}$ acts as a separator. We can decompose 
		each path $\pi_i= \pi_i^1 \circ \pi_i^2$ at $z_i$ such that $\pi_i^1$ starts in $S$ and terminates at 
		$z_i$ and $\pi_i^2$ starts at $z_i$ and terminates in $K$. We find that 
		$\Pi^1=\{\pi_1^1,\ldots,\pi_M^1\}$ is a family of 
		node-disjoint paths such that $S$ is linked in $g_1$ to $Z$. Analogously through $\Pi^2$ we see that 
		$Z$ is linked in $g_2$ to $K$.
		
		Now assume $S$ is linked in $g_1$ to $Z$ and $Z$ is linked in $g_2$ to $K$. 
		Let $Z=(z_1,\ldots,z_M)$ where 
		again $M$ is the cardinality of $S$ and $K$. That means we find a set 
		$\Pi^1=\{\pi_1^1,\ldots , \pi_M^1\}$ of node-disjoint paths such that $\pi_i^1$ starts in $S$ and 
		terminates at $z_i$ and we find a family of node-disjoint paths 
		$\Pi^2=\{\pi_1^2,\ldots, \pi^2_{M}\}$ such 
		that $\pi_i^2$ starts at $z_i$ and terminates in $K$. We can now concatenate the paths $\pi_i:= 
		\pi_i^1\circ \pi_i^2$. Since $\pi_i^1$ only contains nodes from $\mathcal{N}_1$ and $\pi_i^2$ from 
		$\mathcal{N}_2$, it is clear, that the paths from $\Pi:=\{\pi_1 , \ldots, \pi_M\}$ are again 
		node-disjoint. Hence $S$ is linked in $g$ to $K$.
	\end{proof}		
	A special case is the concatenation of a gammoid with its own transpose, $\Gamma \circ \Gamma'$. Due to the 
	symmetry between paths in $\Gamma$ and $\Gamma'$ explained before, we find that an input set $S$ 
	is independent in $\Gamma \circ \Gamma'$ if and only if it is independent in $\Gamma$.
	
\subsection{Gramian Matrix}
	\begin{definition}
		Let $T:\mathbb{C}\to \mathcal{C}^{P \times M}$ be the transfer function of a linear system. 
		The \emph{input gramian} of the system is defined as
		\begin{equation}
			G(\sigma) := T^*(\sigma) T(\sigma) \quad s\in \mathbb{C}
		\end{equation}
		where the asterisk denotes hermitian conjugate.
	\end{definition}
	Let $S=\{s_1,\ldots,s_M \}$ be the input set and $Z=\{z_1,\ldots, z_P\}$ be the output set that belongs 
	to the transfer function $T$.
	Making use of proposition \ref{prop:Transferfunction} we can compute the input gramian via
	\begin{equation}
		G_{ji}(\sigma) = \sum_{k=1}^P  \frac{1}{\vert \sigma\vert^2}
		\sum_{\rho\in\mathcal{P}(s_i,z_k)} \frac{F(\rho)}{{\sigma}^{\text{len}\,\rho}}
		\sum_{\pi\in\mathcal{P}(s_j,z_k)} \frac{F(\pi)}{\bar{\sigma}^{\text{len}\,\pi}} \, .
	\end{equation}
	We already know that if there is a path 
	$\pi$ in $\Gamma$ that goes from $s_j$ to $z_k$, then there is a path $\pi'$ in $\Gamma'$ that goes from 
	$z_k'$ to $s_j'$.
	So if 
	\begin{equation}
		\rho = (n_0\to\ldots \to n_{l-1}\to n_l)
	\end{equation}		
	is a path in $\Gamma$ from $n_0=s_i$ to $n_l = z_k$, and if $\pi'$ 
	\begin{equation}
		\pi' = (p'_0 \to \ldots  \to p'_r)
	\end{equation}		
	is a path in $\Gamma'$ from 
	$p'_0 = z_k'$ to $p'_r = s_j'$, then with respect to the identification $z_k=z'_k$ we can interpret 
	\begin{equation}
		\rho \circ \pi' = (n_0\to\ldots \to n_{l-1}\to p'_0 \to \ldots \to p'_r)
	\end{equation}		
	as a path in $\Gamma\circ \Gamma'$ with
	\begin{equation}
		F(\rho \circ \pi') = F(\rho) F(\pi') = F(\rho)F(\pi) \, .
	\end{equation}
	We can introduce a multi-index notation 
	\begin{equation}
		\sigma^\psi := \sigma^{\text{len}\,\rho} \bar{\sigma}^{\text{len}\,\pi}
	\end{equation}		
	because we know, that any path $\psi$ in $\Gamma \circ \Gamma'$ has always a unique decomposition in such a 
	$\rho$ and $\pi'$. We end up with the formula
	\begin{equation} \label{eq:GramianShort}
		G_{ji}(\sigma) = \frac{1}{\vert s\vert^2} \sum_{\psi \in \mathcal{P}(s_i,s_j')} 
		\frac{F(\psi)}{\sigma^\psi}
	\end{equation}
	Before we turn our interest to the meaning of the gramian for dynamic compressed sensing, we want to 
	provide tools in the form of the following lemma and proposition.
	%
	%
	\begin{lemma} \label{lemma:shortpaths}
		Let $\Gamma=(\mathcal{L},g,\mathcal{M})$ be a gammoid and let $a,b \in\mathcal{L}$ be two input nodes. In 
		$\Gamma \circ \Gamma'$ let $\eta_{aa'}$ denote the shortest path from $a$ to $a'$, $\eta_{bb'}$ the 
		shortest path from $b$ to $b'$ and $\eta_{ab'}$ shall denote the shortest path from $a$ to $b'$. 
		Then
		\begin{equation}
			\frac{\text{len}\, (\eta_{aa'}) + \text{len}\, (\eta_{bb'}) }{2} \leq \text{len}\, (\eta_{ab'}) \, .
		\end{equation}
	\end{lemma} 
	%
	%
	\begin{proof}
		By construction we know that there is a $z\in \mathcal{M}$ and a decomposition 
		$\eta_{ab'}=\alpha \circ \beta'$ such that $\alpha$ goes 
		from $a$ to a $z$ and $\beta$ goes from $b$ to $z$. We now find that $\alpha \circ \alpha'$ goes 
		from $a$ to $a'$. By assumption of the lemma, $\alpha \circ \alpha'$ is not shorter than $\eta_{aa'}$, 
		thus
		\begin{equation}
			\text{len}\, (\eta_{aa'}) \leq \text{len}\,(\alpha \circ \alpha') = 2\, \text{len}\, (\alpha) \, .
		\end{equation}
		Analogously we find
		\begin{equation}
			\text{len} \, ( \eta_{bb'} )\leq 2 \,\text{len}\, (\beta) \, .
		\end{equation}
		We can now add these two inequalities together to get
		\begin{equation}
			  \text{len}\, (\eta_{aa'}) + \text{len}\, (\eta_{bb'} ) \leq 
			 2 (\text{len} \, (\alpha) + \text{len}\, (\beta ) )  \, .
		\end{equation}
		On the right hand side we identify the length of $\eta_{ab'}$ to get
		\begin{equation}
			\text{len}\, (\eta_{aa'} )+ \text{len}\, (\eta_{bb'}) \leq 2 \, \text{len} \,(\eta_{ab'})  \, .
		\end{equation}
	\end{proof}
	%
	%
	%
	\begin{proposition}\label{prop:short}
		Let $T$ be the transfer function of a dynamic system with input set $S=\{s_1,\ldots , s_M\}$ and 
		gammoid $\Gamma$, and let 
		$G=T^* T$ be the input gramian. Consider the quantity
		\begin{equation}
			\mu_{ij}(\sigma) := \frac{\vert G_{ij}(\sigma) \vert}{ \sqrt{ G_{ii}(\sigma)G_{jj}(\sigma)} } \, .
		\end{equation}
		Let $\eta_{ij'}$ be the shortest path in $\Gamma \circ \Gamma$ from $s_i$ to $s_j'$.
		If lemma \ref{lemma:shortpaths} holds with equality, then 
		\begin{equation}
			\lim_{\vert \sigma \vert \to \infty} \mu_{ij}(\sigma) = \frac{\vert F(\eta_{ij'}) 
			\vert}{\sqrt{F(\eta_{ii'})F(\eta_{jj'})}} \, .
		\end{equation}
		If the lemma holds with strict inequality, then
		\begin{equation}
			\lim_{\vert \sigma \vert \to \infty} \mu_{ij}(\sigma) = 0 .
		\end{equation}
		Note, that there can be several shortest paths, so $\eta_{ii'}$ can be a set of paths.
	\end{proposition}
	%
	%
	\begin{proof}
		Say $S=\{s_1,\ldots,s_M\}$ is the input set that leads to the transfer function $T$ and the input 
		gramian $G$. Using equation \eqref{eq:GramianShort} we can write
		\begin{equation} \label{eq:G1}
			\mu_{ij}(\sigma)=\frac{
			\left\vert
			\sum_{\psi\in\mathcal{P}(s_i,s_j')} F(\psi) \sigma^{-\psi} 
			\right\vert
			}{
			\sqrt{			
				\sum_{\pi\in\mathcal{P}(s_i,s_i')} F(\pi)\sigma^{-\pi}
				\sum_{\theta\in\mathcal{P}(s_j,s_j')} F(\theta)\sigma^{-\theta}
			}} \, .
		\end{equation}	
		The terms under the 
		square-root are non-negative. To see that, let $\pi=\alpha \circ \beta'$ be a path from $s_i$ to $s_i'$.
		Then we know, that also $\beta \circ \alpha'$, $\alpha \circ \alpha'$ and $\beta \circ \beta'$ exist and 
		all go from $s_i$ to $s_i'$. Thus, in $G_{ii}$ we always find the four terms
		\begin{equation}
			R:=\frac{F(\alpha \circ \beta')}{\sigma^{\text{len}\,\alpha}\bar{\sigma}^{\text{len}\,\beta}} +
			\frac{F(\beta \circ \alpha')}{\sigma^{\text{len}\,\beta}\bar{\sigma}^{\text{len}\,\alpha}} +
			\frac{F(\alpha \circ \alpha')}{\sigma^{\text{len}\,\alpha}\bar{\sigma}^{\text{len}\,\alpha}} +
			\frac{F(\beta \circ \beta')}{\sigma^{\text{len}\,\beta}\bar{\sigma}^{\text{len}\,\beta}}
		\end{equation}
		together. So it is sufficient to show, that $R$ is non-negative. With 
		$A:= F(\alpha)/\sigma^{\text{len}\,\alpha}$ and $B:= F(\beta)/\sigma^{\text{len}\,\beta}$
		we can rewrite $R$ as
		\begin{equation}
			R =A \bar{B} + B \bar{A} + A \bar{A} + B \bar{B}.
		\end{equation}
		With $A=x+\text{i} y$ and $B=u+\text{i}v$ we find
		\begin{equation}
			R = (x+u)^2(y+v)^2	 \geq 0 \, .	
		\end{equation}
		The same holds for $G_{jj}$. 
	
		We now proceed with \eqref{eq:G1}. As we want to take the limit $\vert \sigma \vert \to \infty$, 
		the smallest 
		powers of $\sigma$ will be dominant. The smallest powers of $\sigma$ correspond to the shortest 
		paths. We neglect higher orders and the asymptotic behaviour 
		\begin{equation}
			\mu_{ij}(\sigma) \simeq \frac{ \left\vert F(\eta_{ij'}) \sigma^{-\eta_{ij'}}  \right\vert}{\sqrt{
							F(\eta_{ii'})F(\eta_{jj'}) \sigma^{-(\eta_{ii'}+\eta_{jj'})}
			}}
		\end{equation}			
		where we use the sign ``$\simeq$'' to denote the asymptotic behaviour for 
		$\vert \sigma \vert \to\infty$. Since the path $\eta_{ii'}$ is always symmetric in the sense 
		$(\eta_{ii'})'=\eta_{ii'}$ we find
		$\sigma^{-\eta_{ii'}}=\vert \sigma \vert^{\text{len}\, (\eta_{ii'})}$. The same holds for 
		$\eta_{jj'}$. With this we 	
		get 
		\begin{equation}
			\mu_{ij}(s) \simeq  \frac{\vert F(\eta_{ij'})\vert }{\sqrt{F(\eta_{ii'})F(\eta_{jj'})}}
			\vert \sigma \vert^{\frac12 \left( \text{len}\, (\eta_{ii'}) + \text{len}\, (\eta_{jj'})\right) 
			- \text{len}\,(\eta_{ij'}) }
			\, .
		\end{equation}
		From lemma \ref{lemma:shortpaths} we know, that the exponent of $\vert s \vert$ is always non-negative, 
		thus the limit always exists. If the lemma holds with equality, then the exponent equals zero and we 
		get
		\begin{equation}
			\mu_{ij}(\sigma) \simeq  \frac{\vert F(\eta_{ij'})\vert }{\sqrt{F(\eta_{ii'})F(\eta_{jj'})}} \, .
		\end{equation}
		If inequality holds, then the exponent of $\vert \sigma \vert$ is negative and we find
		\begin{equation}
			\lim_{\vert \sigma \vert \to \infty}\mu_{ij}(\sigma) = 0 \, .
		\end{equation}				 
	\end{proof}

\section{Spark - Mutual Coherence Inequality}
	As a final result we will show that the spark-mutual coherence inequality that has already been presented for 
	the static problem in \cite{donoho_optimally_2003} stays valid for the gammoid of a linear dynamic system. 
	In the proof for the static problem, one utilizes the Gershgorin circle theorem. 
	
	For the dynamic problem we can do a similar approach locally, meaning we get an inequality at each point 
	$\sigma\in\mathbb{C}$ in the complex plain. We will first show, that the it makes sense to consider the \textit{generic 
	rank} instead of the standard rank of the transfer function $T$. This will help us to finally deduce a global 
	inequality.
\subsection{Eigenvalues of the Gramian}
	Let $T:\mathbb{C}\to \mathbb{C}^{P\times M}$ be the transfer 
	function of a linear dynamic system. One knows, that $T$ has singularities at the eigenvalues of $A$ (see 
	section 6. \ref{sec:Transfer}), but these will cause no issue for the following calculations. 
	Assume for an $\sigma_0\in\mathbb{C}$ we find the rank
	\begin{equation}
		\text{rank}\, T(\sigma) = r < M \, .
	\end{equation}
	If we regard $T(\sigma)=(\vec{t}_1(\sigma),\ldots,\vec{t}_M(\sigma))$ as a set of column vectors, then
	$\{\vec{t}_1(\sigma_0),\ldots, \vec{t}_M(\sigma_0)\}$ is a linearly dependent set of vectors. 
	This can have two reasons. 
	Either, $\{\vec{t}_1,\ldots,\vec{t}_M\}$ as function are linearly dependent, say of rank $r$. 
	Then it is clear, that the rank of $T(\sigma)$ will never exceed $r$. Or, $\{\vec{t}_1,\ldots,\vec{t}_M\}$ is 
	a set of linearly independent functions, and the linear dependence at $\sigma_0$ is just an unfortunate 
	coincidence. In this case, for any $\sigma$ in a vicinity of $\sigma_0$, we will find that the rank of 
	$T(\sigma)$ is $M$ 
	almost everywhere. For this reason, one defines the \emph{structural} \cite{lunze_einfuhrung_2016} or 
	\emph{generic rank} \cite{murota_matrices_2009} of $T$ as
	\begin{equation}
		\text{Rank}\, T := \max_{\sigma\in\mathcal{C}} \text{rank}\, T(\sigma) \, .
	\end{equation}
	From linear algebra it is clear, that for a fixed $\sigma\in\mathbb{C}$, the equation
	\begin{equation}
		T(\sigma) \tilde{\vec{w}}(\sigma) = \tilde{\vec{y}}(\sigma)
	\end{equation}
	can be solved for $\tilde{w}(\sigma)$ if the rank of $T(\sigma)$ equals $M$. From the argumentation above 
	it becomes clear, that a generic rank of $M$ already renders the whole system invertible.
	
	Now, we formulate the well known Gershgorin theorem \cite{gershgorin_uber_1931} in 
	a form suitable for our purpose.
	%
	%
	\begin{lemma}\label{lemma:zeroeigen}
		Consider a complex matrix $G:\mathbb{C}\to\in\mathbb{C}^{M\times M}$.
		We call $G$ strict diagonal dominant in $\sigma$ if for all $i=1,\ldots,M$ we find
		\begin{equation}
			\vert G(\sigma)_{ii} \vert > \sum_{j\neq i} \vert G(\sigma)_{ij}\vert \, .
		\end{equation}		 
	If $G$ is strict diagonal dominant in $\sigma$, then zero is not an eigenvalue of $G(\sigma)$.
	\end{lemma}
	%
	%
	\begin{proof}
		For a fixed $\sigma$ set $A:=G(\sigma)$. 
		Assume $\lambda\in\mathbb{C}$ is an eigenvalue of $A$ with eigenvector $\vec{v}\in\mathbb{C}^M$. Let 
		$v_i$ 
		be a component of $\vec{v}$ of maximal magnitude, i.e. $\vert v_i \vert \geq \vert v_j \vert$ for all 
		$j=1,\ldots,M$. Without loss of generality we assume $v_i=1$. We write the eigenvalue equation in 
		components
		\begin{equation}
			\sum_{j=1}^M A_{ij}v_j = \lambda v_i \, .
		\end{equation}
		We extract $j=i$ from the sum to get
		\begin{equation}
			\sum_{j\neq i} A_{ij} v_j = \lambda - A_{ii} \, .
		\end{equation}
		Taking the absolute value and applying the triangle inequality yields
		\begin{equation}
			\vert \lambda - A_{ii} \vert \leq \sum_{j\neq i} \vert A_{ij} \vert \vert v_j \vert \, .
		\end{equation}
		By constriction $\vert v_j \vert \leq \vert v_i \vert =1$ thus we get the final result
		\begin{equation}
			\vert \lambda - A_{ii} \vert \leq \sum_{j\neq i} \vert A_{ij} \vert \, .
		\end{equation}
		With $r_i= \sum_{j\neq i} \vert A_{ij} $ we can give a boundary of the spectrum $\text{spec}\,(A)$ via a 
		union of circles in the complex plane,
		\begin{equation}
			\text{spec}\,(A) \subseteq \bigcup_{i=1}^M 
			\{ \lambda \in \mathbb{C} \, | \, \vert \lambda - A_{ii} \vert \leq r_i   \, .
			\}
		\end{equation}
		
		Note that a strict diagonal dominant matrix necessarily has diagonal elements $\vert A_{ii} \vert > 0$-
		Furthermore $A_{ii}>r_i$ for each $i=1,\ldots, M$ and thus
		\begin{equation}
			0 \not\in \{ \lambda \in \mathbb{C} \, | \, \vert \lambda - A_{ii} \vert \leq r_i   \}  \, .
		\end{equation}
		Consequently $0 \not\in \text{spec}\,(A)$.
	\end{proof}
	Note, that the diagonal element of an input gramian $G_{ii}$ is given by 
	\begin{equation}
			G_{ij}(\sigma) = \sum_{j=1}^P T^*_{ij}(\sigma) T_{ji}(\sigma) = \sum_{j=1}^P \vert T_{ji}(\sigma) 
			\vert^2	
	\end{equation}
	and thus it is the zero function if and only if each $T_{ji}$ for $j=1,\ldots,P$ is the zero function. 
	Again consider the dynamic system
	\begin{equation}
			 T(\sigma) \tilde{\vec{w}}(\sigma) = \tilde{\vec{y}}(\sigma) \, .
	\end{equation}
	If the $i$-th column of $T(\sigma)$ is zero for all $\sigma\in\mathbb{C}$, then 
	$\tilde{w}_i$ has no influence on 
	the output at all. 
	It is therefore trivially impossible to gain any information about $\tilde{w}_i$. We want to 
	exclude this trivial case from our investigation and henceforth assume, that the diagonal elements of 
	$G$ are non-zero. By construction it is clear, that each diagonal element $G_{ii}(\sigma)$ is indeed 
	a positive real number.
	
\subsection{Strict Diagonal Dominance}
	Consider a linear dynamic input-output system with input set 
	$S = \{s_1,\ldots, s_M\}$ and input gramian $G:\mathbb{C} 
	\to \mathbb{C}^{M \times M}$ and recall the definition of the 
	coherence between input nodes $s_i$ and $s_j$
	\begin{equation}
		\mu_{ij}(\sigma) := \frac{\vert G_{ij}(\sigma) \vert}{\sqrt{G_{ii}(\sigma)G_{jj}(\sigma)}} \, .
	\end{equation}
	The mutual coherence at $\sigma$ is defined as
	\begin{equation}
		\mu(\sigma) := \max_{i\neq j} \mu_{ij}(\sigma)  \, .
	\end{equation}
	For the special case, that the transfer function $T$ is constant and real, this coincides with the 
	definition of the mutual coherence from \cite{donoho_optimally_2003}. Also from there we take the following 
	theorem and show that it stays valid for the dynamic problem.
	%
	%
	\begin{proposition}\label{prop:DiagonalDominance}
		Consider a linear dynamic input-output system with input set $S=\{s_1,\ldots,s_M\}$, input gramian 
		$G$. Let $\sigma\in\mathbb{C}$ and and $\mu(\sigma)$ the mutual coherence in $\sigma$.
		If the inequality 
		\begin{equation}
			G_{ii}(s) > \mu(s)
		\end{equation}
		holds for $i=1,\ldots, M$, then $G$ is strictly diagonally dominant at $s$.
	\end{proposition}
	%
	%
	\begin{proof}
	First, rescale the system by
	\begin{equation}
			G(\sigma) \leadsto \frac{G(\sigma)}{\text{tr}\, G(\sigma)} 
	\end{equation}			
	which gives the property
	\begin{equation}
		\sum_{i=1}^M G_{ii}(\sigma) = 1 \, . \label{eq:G2}
	\end{equation}
	Note, that $G_{ii}(\sigma)>0$ is already clear and from the equation above 
	$G_{ii}<1$. The case $M=1$ is trivial.
	By definition for all $i,j=1,\ldots,M$ with $i \neq j$
	\begin{equation}
		\vert G_{ij}(\sigma) \vert \leq \mu(\sigma) \sqrt{G_{ii}(\sigma)G_{jj}(\sigma)} \,.
	\end{equation}		
	By assumption of the proposition $\mu(\sigma) \leq G_{ii}(\sigma)$ for all $i=1,\ldots,M$, and since all 
	quantities are non-negative
	\begin{equation}
		\mu < \sqrt{G_{ii}(\sigma)G_{jj}(\sigma)} 
	\end{equation}
	for all $i\neq j$.
	Combine the latter two inequalities and sum over all $j=1,\ldots,$ $i-1,i+1,\ldots,M$ to get 
	\begin{equation}
		\sum_{j\neq i} \vert G_{ij}(\sigma) < G_{ii}(\sigma) \sum_{j\neq i} G_{jj}(\sigma) 
	\end{equation}
	for all $i,\ldots,M$.
	Due to equation \eqref{eq:G2}, the sum on the right hand side is smaller than one, thus for all $i=1,
	\ldots , M$ we find
	\begin{equation}
		\sum_{j\neq i} \vert G_{ij}(\sigma) \vert < G_{ii}(\sigma) 
	\end{equation}
	which is exactly the definition of strict diagonal dominance at $\sigma$.
	\end{proof}

\subsection{A note on the Coherence}
	In contrast to the static problem, where the coherence is a constant, we have here a function in the 
	complex plane. Say, we have a small input set $S=\{s_1,s_2\}$. It is possible, that $s_1$ and $s_2$ 
	are coherent in one regime of 
	$\mathbb{C}$ but will be incoherent in another. Proposition \ref{prop:DiagonalDominance} and 
	lemma \ref{lemma:zeroeigen} show, that a small coherence in a single $\sigma \in \mathbb{C}$ is sufficient 
	to get a high generic rank, and by this invertibility of the system.
	To have a measure whether two nodes $s_i$ and $s_j$ are distinguishable somewhere in $\mathbb{C}$ it 
	makes sense to define the \textit{global coherence} 
	\begin{equation}
		\mu_{ij} := \inf_{\sigma\in\mathbb{C}} \mu_{ij}(\sigma) \, .
	\end{equation}
	It is easy to see that the inequality
	\begin{equation}
		\max_{i\neq j} \inf_{\sigma\in\mathbb{C}} \mu_{ij}(\sigma) \leq   \inf_{\sigma\in\mathbb{C}} 
		\max_{i\neq j} \mu_{ij}(\sigma)
		\label{eq:G3}
	\end{equation}
	holds, formulated in terms of the global coherence $\mu_{ij}$ and mutual coherence $\mu(\sigma)$ 
	\begin{equation}
		\max_{i\neq j} \mu_{ij} \leq \inf_{\sigma\in\mathbb{C}} \mu(\sigma) \, . 
	\end{equation}
	Proposition \ref{prop:DiagonalDominance} holds for the mututal coherence $\mu(\sigma)$ at any 
	$\sigma\in\mathbb{C}$. The least restrictive bound for strict diagonal dominance is obviously achieved for 
	a small mutual coherence $\mu(\sigma)$, so we would like to compute the minimum or infimum of $\mu(\sigma)$,
	the \emph{global mutual coherence}
	\begin{equation}
		\mu := \inf_{\sigma \in \mathbb{C}}  \mu(\sigma) \, .
	\end{equation}	 
	The inequality above indicates, that small global coherences $\mu_{ij}$ do not necessarily lead to a 
	small global mutual coherence $\mu$. More precisely
	\begin{equation}
		\max_{i\neq j}\mu_{ij} \leq \mu \, .
	\end{equation}
 
 	As an example, consider three input nodes $s_1$, $s_2$ and $s_3$ and 
	three subsets of the complex plane $U,V,W\subseteq \mathbb{C}$. It might happen that $\mu_{12}(\sigma)|_{U}$ 
	is small, let us assume it vanishes, and so do $\mu_{23}(\sigma)|_{V}$ and $\mu_{13}(\sigma)|_W$. Thus, 
	all global coherences are vanishingly small. However, if $U$, $V$ and $W$ do not intersect, the 
	global coherences are 
	attained at different regions in the complex plane which means, in $U$ we can distinguish $s_1$ from 
	$s_2$ but 
	we cannot distinguish $s_2$ from $s_3$, and the same for $V$ and $W$.
	
	\subsubsection{Shortest Path Coherence}
		We want to estimate the global mutual coherence by the
		\emph{mutual shortest path coherence}
		\begin{equation}
			\mu^\text{short} := \lim_{\vert \sigma \vert\to \infty} \mu(\sigma) \, .
		\end{equation}
		which is clearly an upper bound for the global mutual coherence.		
		In contrast to the non-commuting infimum and maximum operations, 
		the proposition below shows, that the limit and the maximum operations commute. 
		We find
		\begin{equation}
			\lim_{\vert \sigma\vert\to \infty} \max_{i\neq j} \mu_{ij}(\sigma) =
			\max_{i\neq j} \lim_{\vert \sigma\vert\to \infty}  \mu_{ij}(\sigma) = \max_{i \neq j} 
			\mu_{ij}^\text{short} 
			\, .
		\end{equation}
		The quantity on the right hand side, $\mu_{ij}^\text{short}$ is the \emph{shortest path coherence} 
		between $s_i$ and $s_j$, which already appeared in proposition \ref{prop:short}.
		%
		%
		\begin{proposition}
			Let $\mu_{ij}(\sigma)$ be the coherence of $s_i$ and $s_j$. Then
			\begin{equation}
				 \lim_{\vert \sigma\vert\to \infty} \max_{i\neq j} \mu_{ij}(\sigma) = 
				 \max_{i\neq j} \lim_{\vert \sigma\vert\to \infty} \mu_{ij}(\sigma) \, .
			\end{equation}
		\end{proposition}
		%
		%
		\begin{proof}
			First note, that $\mu_{ij}(\sigma)$ is continuous, $0 \leq \mu_{ij}(\sigma) \leq 1$ and that 
			any singularity of $\mu_{ij}(\sigma)$ is removable. 	
			Furthermore we have already seen that each $\mu_{ij}(\sigma)$ convergent with limit 
			$\mu^\text{short}_{ij}$ as $\vert \sigma \vert \to \infty$. Due to these convenient properties, 
			the following setting is sufficient.
			
			Let $f_a:\mathbb{R}_{\geq 0} \to \mathbb{R}$ a family of continuous and 
			bounded functions with $a \in I$ where 
			$I$ is a finite index set. We want to show that
			\begin{equation}
				\lim_{x \to \infty } \max_{a\in I} f_a(x) = \max_{a\in I} \lim_{x\to\infty} f_a(x) \, .
			\end{equation}
			To see that, let $M_a:=\lim_{x\to\infty} f_a(x)$ and let $a^*\in I$ such that 
			$M_{a^*}=\max_{a\in I}M_a$. By definition of $M_a$, for any $\epsilon>0$ there is an $x_a$ such 
			that for all $x>x_a$ we have
			\begin{equation}
				\vert f_a(x) - M_a\vert < \epsilon \, .
			\end{equation}
			Set $x_0 := \max_{a\in I} x_a$ such that we can use the same epsilon and $x_0$ for all indices $a$. 
			Let us divide $I= I_0 \dot{\cup} I_1$ such that $I_0$ contains all indices with $M_a < M_{a^*}$ and 
			$I_1$ contains this indices with $M_a = M_{a^*}$

			Let us first consider all $I_0$ and let us choose $\epsilon$ such that
			\begin{equation}
				\epsilon < \frac{1}{2} \vert M_{a^*} - M_{a} \vert
			\end{equation}
			for all $a\in I_0$. We can rewrite this as
			\begin{equation}
				M_a + \epsilon < M_{a^*} - \epsilon \, .
			\end{equation}
			For this choice of $\epsilon$ there is an $x_0$ such that $\vert f_a(x) - M_a \vert < \epsilon$ 
			for $x>x_0$, i.e., 
			\begin{equation}						
				f_a(x)\in (M_a - \epsilon, M_a + \epsilon)
			\end{equation}
			an analogously
			\begin{equation}
				f_{a^*}(x)\in (M_{a^*}-\epsilon, M_{a^*}+\epsilon) \, .
			\end{equation}
			Due to our choice of epsilon these two intervals are disjoint and one can see that 
			$f_a(x)< f_{a^*}(x)$ for all $x > x_0$.	Thus for $x$ large enough we can always neglect $I_0$, 
			\begin{equation}
				\max_{a\in I} f_a(x) = \max_{a\in I_1} f_a(x) \, .
			\end{equation}
	
			Let us now focus on $I_1$. Consider 
			\begin{equation}
				r_a(x):=\vert f_a(x) - M_{a^*} \vert 
			\end{equation}
			for $a\in I_1$ and let $a'\in I_1$ such that $f_{a'}(x) = \max_{a\in I_1}f_a(x)$. It is now clear 
			that
			\begin{equation}
				 r_{a'}(x) \leq \max_{a \in I_1} r_a(x) 
			\end{equation}
			for all $x$. Insertion of the definitions yields
			\begin{equation}
				  \left\vert f_{a'}(x) - M_{a^*} \right\vert 
				  \leq \max_{a \in I_1} \vert f_a(x) - M_{a^*} \vert 
			\end{equation}
			and for $x> x_0$ we deduce
			\begin{equation}
			  \left\vert f_{a'}(x) - M_{a^*} \right\vert < \epsilon \, .
			\end{equation}
			Since we can do that for arbitrary small $\epsilon>0$ we find the convergence
			\begin{equation}
				\lim_{x\to\infty} f_{a'}(x) = M_{a^*} \, .
			\end{equation}
			We can now insert the definitions 
			\begin{equation}
				f_{a'}(x)=\max_{a\in I_1} f_a(x) = \max_{a\in I} f_a(x)
			\end{equation}	
			and		
			\begin{equation}
				M_{a^*}=\max_{a\in I} M_a = \max_{a\in I} \lim_{x\to \infty} f_a(x)
			\end{equation}						
			to get the desired result
			\begin{equation}
				\lim_{x\to\infty}   \max_{a\in I} f_a(x) =  \max_{a\in I} \lim_{x\to \infty} f_a(x) \, .
			\end{equation}
			
		\end{proof}
	\subsubsection{Feed-Forward Graphs}
		As one special case we want to consider systems with a feed-forward structure. Such a structure 
		for instance appears in artificial neural networks. From the matroid theory side, such systems correspond 
		to cascade gammoids which were investigated in \cite{mason_class_2018}. 
		
		A gammoid $\Gamma=(\mathcal{L},g,\mathcal{M})$ is called a \emph{cascade} if it has the following 
		structure.		
		The graph $g=(\mathcal{N},\mathcal{E})$ consists of a dijsoint union of $L+1$ node sets
		\begin{equation}
			\mathcal{N} = \mathcal{N}_0 \dot{\cup} \ldots \dot{\cup} \mathcal{N}_L
		\end{equation}				
		called layers. With 
		$(i,l)$ we mean node $i$ from layer $\mathcal{N}_l$. 
		Each edge in $\mathcal{E}$ has the form
		\begin{equation}
			(i,l) \to (j,l+1) \, .
		\end{equation}	
		The input and output ground sets are set to be $\mathcal{L}=\mathcal{N}_0$ and 
		$\mathcal{M}=\mathcal{N}_L$.		
		
		In a cascade, all paths from the input layer to the output layer have the same length. For 
		$(i,0)\in\mathcal{L}$ and $(j,L)\in\mathcal{M}$ we find the 
		transfer function to be
		\begin{equation}
			T_{ji}(\sigma) = \frac{1}{\sigma^L} \sum_{\pi\in \mathcal{P}((i,0),(j,L))} F(\pi)
		\end{equation}
		One can see that
		\begin{equation}
			A_{ji} := \sum_{\pi\in \mathcal{P}((i,0),(j,L))} F(\pi)
		\end{equation}		
		defines the entries of a real constant matrix $A$.		
		The transfer function achieves the simple form
		\begin{equation}
			T(\sigma) = \frac{1}{\sigma^L} A \, 
		\end{equation}
		and the coherence $\mu_{ij}(\sigma)$ turns out to be a constant
		\begin{equation}
			\mu_{ij}(\sigma) = \frac{\langle a_i , a_j \rangle}{\Vert a_i \Vert \Vert a_j \Vert} \, .
		\end{equation}
		The latter equation looks exactly like the coherence one defines for the static compressed 
		sensing problem \cite{donoho_optimally_2003}. It follows, that for the class of cascade gammoids the 
		different notions of coherence all coincide, and consequently the shortest path coherence indeed yields 
		a method to compute the mutual coherence exactly.
	
\subsection{Estimation of the Spark}
		Also this calculation is in analogy to the static problem from \cite{donoho_optimally_2003}.		
		Let $G$ be the gramian of a linear dynamic input-output system with input set $S=\{s_1,\ldots,s_M\}$ 
		and let $\mu(\sigma)$ be its mutual coherence at $\sigma\in\mathbb{C}$. By the definition
		\begin{equation}
			\vert G_{ij}(\sigma) \vert \leq \mu(\sigma) \sqrt{G_{ii}(\sigma)G_{jj}(\sigma)} \, .
		\end{equation}
		Now let $G_{kk}(\sigma)$ be the maximum of all $G_{ii}(\sigma)$ for $i=1,ldots,M$. We replace the 
		square root and sum over all $j=1,\ldots , i-1,i+1,\ldots,M$ to get
		\begin{equation}
		 		\sum_{j \neq i}\vert G_{ij}(\sigma) \vert \leq (M-1) \mu(\sigma)G_{kk}(\sigma) \, ,
		\end{equation}
		where we used that the right hand side is independent of $j$. The gramian is strict 
		diagonal dominant if for all $i=1,\ldots,M$ we have
		\begin{equation}
			 (M-1) \mu(\sigma)G_{kk}(\sigma) < G_{ii}(\sigma) \, .
		\end{equation}
		The latter inequality is therefore sufficient for strict diagonal dominance at $\sigma$ and 
		consequently to invertibility of the dynamic system. We therefore proceed with this inequality by using 
		$G_{ii}(\sigma)/G_{kk}(\sigma) < 1$ to get
		\begin{equation}
			M < \frac{1}{\mu(\sigma)} + 1 \, . \label{eq:G4}
		\end{equation}
		Note, that the gammoid $G$ depends on the choice of the input set $S\subseteq \mathcal{L}$. However, 
		the sufficient condition for strict diagonal dominance only takes $M=\text{card}\, S$ into account. It 
		becomes clear, that for all $\tilde{S}\subseteq \mathcal{L}$ with $\text{card}\,\tilde{S}\leq M$ we get 
		strict diagonal dominance. Since strict diagonal dominance leads to invertibility, it also 
		tells us that $\tilde{S}$ is independent $\Gamma$. By definition of the spark, $\text{spark}\, \Gamma$ 
		is the largest integer such that $\tilde{S}$ is independent in $\Gamma$ whenever 
		$\text{card}\,\tilde{S}\leq \text{spark}\,\Gamma$. Since any $M$ that fulfils \eqref{eq:G4} leads to 
		independence, it becomes clear, that the spark is not smaller than the right hand side of 
		this inequality. 
		Therefore
		\begin{equation}
			\text{spark} \, \Gamma \geq \frac{1}{\mu(\sigma)} + 1 \, .
		\end{equation}
		The tightest bound for the spark is attained for the global mutual coherence $\mu$. 
		In practice, however, we have to use the limit $\vert \sigma \vert \to \infty$  
		and rely on the shortest path coherence.

\bibliographystyle{plain}
\bibliography{bib_error_localization}

\end{document}